\documentclass[11pt, a4paper]{amsart}

\usepackage[T1]{fontenc}
\usepackage{microtype}

\usepackage{amsmath}								
\usepackage{amssymb}
\usepackage{amsthm}
\usepackage{amscd}
\usepackage{amsfonts}
\usepackage{stmaryrd}

\usepackage{extarrows}

\usepackage[bookmarks=true, colorlinks=true, linkcolor=blue!50!black,
citecolor=orange!50!black, urlcolor=orange!50!black, pdfencoding=unicode]{hyperref}
\usepackage{color}

\usepackage{tikz}									
\usetikzlibrary{matrix}
\usetikzlibrary{patterns}
\usetikzlibrary{matrix}
\usetikzlibrary{positioning}
\usetikzlibrary{decorations.pathmorphing}
\usetikzlibrary{cd}

\usepackage{nicefrac}
\usepackage{fullpage}
\usepackage{enumerate}

\newtheorem*{theorem*}{Theorem}

\newtheorem{maintheorem}{Theorem}[section]

\newtheorem{theorem}{Theorem}[section]

\newtheorem{proposition}[theorem]{Proposition}
\newtheorem{corollary}[theorem]{Corollary} 

\newtheorem{conjecture}[theorem]{Conjecture}
\theoremstyle{definition}
\newtheorem{definition}[theorem]{Definition}
\newtheorem{example}[theorem]{Example}

\newtheorem{remark}[theorem]{Remark}

\newtheoremstyle{myitemstyle}						
	{}			
	{}			
	{}			
	{}			
	{}			
	{.}			
	{ }			
	{}			
\theoremstyle{myitemstyle}
\newtheorem{myitemthm}{}

\newcommand{\martin}[1]{{\color{green!60!black} \sf Martin: [#1]}}
\newcommand{\martincolor}[1]{{\color{green!60!black} \sf #1}}

\newcommand{\Dmitry}[1]{{\color{blue} \sf Dmitry: [#1]}}
\newcommand{\dmitrycolor}[1]{{\color{blue} \sf #1}}
\newcommand{\yoav}[1]{{{\color{teal} \sf Yoav: [#1]}}}

\newcommand{\R}{\mathbb{R}}

\newcommand{\Z}{\mathbb{Z}}
\newcommand{\ZZ}{\mathbb{Z}}

\newcommand{\C}{\mathbb{C}}
\newcommand{\CC}{\mathbb{C}}

\newcommand{\RR}{\mathbb{R}}

\newcommand{\calA}{\mathcal{A}}

\newcommand{\calF}{\mathcal{F}}

\newcommand{\frakA}{\mathfrak{A}}
\newcommand{\frakG}{\mathfrak{G}}
\newcommand{\calH}{\mathcal{H}}
\newcommand{\calHbar}{\overline{\mathcal{H}}}

\newcommand{\calM}{\mathcal{M}}

\newcommand{\calO}{\mathcal{O}}

\newcommand{\calS}{\mathcal{S}}
\newcommand{\calT}{\mathcal{T}}

\newcommand{\calU}{\mathcal{U}}

\newcommand{\calX}{\mathcal{X}}

\DeclareMathOperator{\Spec}{Spec}
\DeclareMathOperator{\Hom}{Hom}

\DeclareMathOperator{\Aut}{Aut}

\DeclareMathOperator{\val}{val}
\DeclareMathOperator{\characteristic}{char}

\DeclareMathOperator{\Ram}{Ram}

\DeclareMathOperator{\ord}{ord}

\DeclareMathOperator{\PL}{\mathcal{PL}}

\newcommand{\an}{\mathrm{an}}
\newcommand{\ext}{\mathrm{ext}}
\newcommand{\dil}{\mathrm{dil}}

\newcommand{\trop}{{\mathrm{trop}}}
\newcommand{\et}{\mathrm{\acute{e}t}}

\newcommand{\src}{{\mathrm{src}}}
\newcommand{\tar}{{\mathrm{tar}}}

\newcommand{\Ga}{\Gamma}
\newcommand{\ga}{\gamma}

\newcommand{\plC}{\scalebox{0.8}[1.3]{$\sqsubset$}}
\newcommand{\ph}{\varphi}

\title{Abelian tropical covers} 

\date{}

\author{Yoav Len}
\address{Mathematical Institute, University of St Andrews, St Andrews KY16 9SS, UK}
\email{yoav.len@st-andrews.ac.uk}

\author{Martin Ulirsch}
\address{Institut f\"ur Mathematik, Goethe--Universit\"at Frankfurt,
60325 Frankfurt am Main, Germany}
\email{ulirsch@math.uni-frankfurt.de}

\author{Dmitry Zakharov}
\address{Department of Mathematics, Central Michigan University, Mount Pleasant, MI 48859, USA}
\email{dvzakharov@gmail.com}

\begin{document}

\maketitle

\begin{abstract} 
Let $\mathfrak{A}$ be a finite abelian group. In this article, we classify harmonic $\mathfrak{A}$-covers of a tropical curve $\Gamma$ (which allow dilation along edges and at vertices) in terms of the cohomology group of a suitably defined sheaf on $\Gamma$. We give a realizability criterion for harmonic $\mathfrak{A}$-covers by patching local monodromy data in an extended homology group on $\Gamma$. As an explicit example, we work out the case $\mathfrak{A}=\mathbb{Z}/p\mathbb{Z}$ and explain how realizability for such covers is related to the nowhere-zero flow problem from graph theory.
\end{abstract}

\setcounter{tocdepth}{1}
\tableofcontents


\section*{Introduction}

One of the starting points of tropical geometry is the observation that there is a deep analogy between the classical geometry of Riemann surfaces and the geometry of metric graphs, or more generally, (abstract) tropical curves. 

Let $X$ be a Riemann surface and let $B\subseteq X$ be a finite set. Ramified covers $X'\rightarrow X$ that are branched over $B$ are topological coverings of $X_0=X\setminus B$, and the Galois correspondence classifies such covers in terms of the fundamental group $\pi_1(X_0,x_0)$ for some base point $x_0\in X_0$. This beautiful and classical story is explained in many standard textbooks on Riemann surfaces, such as  \cite{Miranda, Szamuely, CavalieriMiles}. In particular, given a finite group $\frakG$, Galois covers with deck group $\frakG$ (not necessarily connected) are in one-to-one correspondence with monodromy representations $\pi_1(X_0,x_0)\to \frakG$. If $\frakG=\frakA$ is abelian, the universal coefficient theorem implies that the set of such covers is equal to
\begin{equation}\label{eq_cohomology}
\Hom\big(\pi_1(X_0,x_0),\frakA\big)\simeq \Hom\big(H_1(X_0,\ZZ),\frakA\big)\simeq H^1(X_0,\frakA).
\end{equation}
Replacing $H^1$ and $\pi_1$ with their \'etale counterparts, this correspondence holds over any algebraically closed field $k$ whose characteristic is zero or relatively prime to $|\frakA|$.





The natural tropical analogue of a non-constant holomorphic map of Riemann surfaces is a \emph{finite harmonic morphism} $\Gamma'\rightarrow\Gamma$ of metric graphs (or tropical curves), which is a continuous map with finite fibers that pulls back harmonic functions on open subsets of $\Gamma$ to harmonic functions on their preimages in $\Gamma'$. Contrary to the algebraic case, a harmonic morphism need not be a topological covering map (even after finitely many points are removed), as harmonic morphisms allow for \emph{dilation along edges}. Namely, via the natural identification of edges with real intervals, the restriction of a harmonic morphism $\phi:\Gamma'\rightarrow\Gamma$ to an edge $e'\subset\Ga'$ is given by 
\[
[0,a]\longrightarrow [0,d\cdot a],\qquad
x\longmapsto d\cdot x.
\]
The coefficient $d\in\Z_{>0}$ is known as the \emph{dilation factor} of $\phi$ along $e'$. The behavior of a harmonic morphism at a vertex $v'\in \Ga'$ is controlled by another phenomenon that we call \emph{dilation at vertices}, which assigns a dilation factor to each vertex as well (see Section~\ref{sec:definitions} below). We also note that dilation should not be confused with the distinct phenomenon of \emph{ramification} for morphisms of weighted graphs, which we discuss at the end of Section~\ref{sec:definitions}.

Dilation phenomena are inherent properties of morphisms of metric graphs, and arise naturally in tropicalization constructions. For this reason, the fundamental group of a metric graph (specifically, its underlying topological space) cannot be used to classify its harmonic covers, and this classification problem is, to the best of our knowledge, currently open. 

\subsection*{Classification of abelian tropical covers}

Our first goal in this article is to classify \emph{abelian harmonic covers} of a fixed metric graph $\Ga$. Given a finite group $\frakG$, a \emph{harmonic $\frakG$-cover} of $\Ga$ is a harmonic morphism $\phi:\Ga'\to \Ga$ together with a fiberwise $\frakG$-action, such that the dilation factor of $\phi$ at a point $p'\in \Ga'$ is equal to the order of its stabilizer group. If $\frakG=\frakA$ is abelian, then $\phi$ admits a convenient cohomological description. Namely, for any $p\in \Ga$ the stabilizer groups of two points of $\phi^{-1}(p)$ are equal, hence the cover determines a family of subgroups $D(p)\subseteq \frakA$ indexed by $p\in \Ga$, an object which we call the $\frakA$-\emph{dilation datum} of the harmonic cover. Choosing a graph model for $\Gamma$, the $\frakA$-dilation datum $D$ determines (by taking quotients) a sheaf of abelian groups $\frakA_D$ on $\Ga$ that we call the \emph{codilation sheaf}. 

\begin{maintheorem}[Theorem \ref{prop_torsor=harmonicGcover}]\label{mainthm_classification}
Let $\Gamma$ be a metric graph or tropical curve, let $\frakA$ be a finite abelian group, and let $D$ be an $\frakA$-dilation datum on $\Gamma$. There is a natural bijection between the sheaf cohomology group $H^1(\Gamma,\frakA_D)$ and the set of harmonic $\frakA$-covers with $\frakA$-dilation datum $D$. 
\end{maintheorem}

We refer to $H^1(\Gamma,\frakA_D)$ as the \emph{dilated cohomology group} of $\Ga$ with respect to the $\frakA$-dilation datum $D$. One may consider Theorem~\ref{mainthm_classification} as a first step towards a tropical analogue of geometric class field theory. 


\subsection*{From algebraic to tropical covers (and back again)} There is a natural tropicalization procedure that associates to a finite cover $F\colon X'\rightarrow X$ of smooth projective algebraic curves over a non-Archimedean field a harmonic morphism $\phi\colon \Gamma_{X'}\rightarrow\Gamma_X$ between the dual tropical curves. In the literature one may find at least two ways to describe this process: one by restricting the associated map $F^{\an}\colon X'^{\an}\rightarrow X^{\an}$ of Berkovich analytic spaces to the non-Archimedean skeletons, as in \cite{ABBRI,ABBRII}, the other from a moduli-theoretic point of view, as in \cite{CavalieriMarkwigRanganathan_tropadmissiblecovers}, using the moduli space of admissible covers. In Section \ref{sec:tropicalization} below we recall the latter approach, paying extra attention to the role of a finite automorphism group $\frakG$. In particular, we describe how to associate to a $\frakG$-cover $F\colon X'\rightarrow X$ of algebraic curves a harmonic $\frakG$-cover $\phi\colon\Gamma_{X'}\rightarrow\Gamma_X$ of tropical curves. 

Describing finite harmonic covers that arise as tropicalizations of finite algebraic covers is a highly non-trivial task, known as the \emph{realizability problem}. 
We refer the reader to \cite{Caporaso_gonality} and \cite{CavalieriMarkwigRanganathan_tropadmissiblecovers} for details, including the connection to the still-open \emph{Hurwitz existence problem} from the classical topology of Riemann surfaces (see \cite{PervovaPetronio_HurwitzexistenceI} for a survey). In the abelian case, however, this problem admits a convenient homological solution, which we describe in Section \ref{section_realizability}. Given a tropical curve $\Ga$ and a finite abelian group $\frakA$, we introduce the \emph{extended homology group} $H_1^{\ext}(\Ga,\frakA)$ whose elements encode local monodromy data of harmonic $\frakA$-covers of $\Ga$. In particular, a class $\eta \in H_1^{\ext}(\Gamma,\frakA)$ determines an associated $\frakA$-dilation datum $D_\eta$, and the realizable covers are exactly the ones that have such $\frakA$-dilation data:

\begin{maintheorem}[Theorem~\ref{thm:main}]\label{mainthm_realizability}
A harmonic $\frakA$-cover $\Ga'\to \Ga$ of tropical curves is realizable over a non-Archimedean field of residue characteristic zero or coprime to $\vert \frakA\vert$ if and only if its $\frakA$-dilation datum is associated to a class in the extended homology group $H_1^{\ext}(\Gamma,\frakA)$.
\end{maintheorem}

In Section~\ref{sec:nowherezeroflowproblem}, we specialize to the case of cyclic covers of prime order. It turns out that our realizability criterion is closely related to the so-called nowhere-zero flow problem from graph theory. In particular, Tutte's 5-flow conjecture has an equivalent formulation in terms of the existence of everywhere-dilated $\Z/5\Z$-covers.

We briefly mention how our results may generalize to the case of a non-abelian group $\frakG$. A harmonic $\frakG$-cover $\Ga'\to \Ga$ determines the structure of a \emph{graph of groups} on a model of $\Ga$, and Bass--Serre theory classifies such covers in terms of an appropriately generalized fundamental group~\cite{Serre_trees, Bass}. However, there is no convenient generalization of the homological realizability criterion, and the difficulties stemming from the Hurwitz existence problem cannot be avoided.

\subsection*{Earlier and related works} Graphs and tropical curves with group actions have been studied by a number of authors. The simplest example is the case of tropical hyperelliptic curves, which are $\ZZ/2\ZZ$-covers of a tree (see \cite{2009BakerNorine}, \cite{2013Chan}, \cite{Caporaso_gonality}, \cite{ABBRII}, \cite{2016Panizzut}, \cite{2017BologneseBrandtChua}, \cite{2017Len}). Expanding on this, Brandt and Helminck \cite{2017BrandtHelminck} consider arbitrary cyclic covers of a tree. Helminck \cite{2017Helminck} looks at the tropicalization of arbitrary abelian covers of algebraic curves from a non-Archimedean perspective, as in \cite{ABBRI, ABBRII}. Our Section \ref{sec:tropicalization} provides a moduli-theoretic approach to the same topic (with possibly non-abelian group) in the spirit of \cite{CavalieriMarkwigRanganathan_tropadmissiblecovers}.

In a different direction, Jensen and Len \cite{2018JensenLen} consider $\ZZ/2\ZZ$-covers of arbitrary tropical curves, and define the tropical Prym variety associated to such a cover. This object is equipped with a canonical polyhedral decomposition, leading to a combinatorial formula for its volume~\cite{2022LenZakharov,2023GhoshZakharov}. A tropical version of Donagi's $n$-gonal construction is investigated in~\cite{2022RoehrleZakharov}. Applications to algebraic Prym--Brill--Noether theory are studied in~\cite{LenUlirsch} and~\cite{creech2022prym}. See \cite{2022Len} for a  survey on tropical Prym varieties. In a similar vein, Song \cite{Song_Ginvariantlinearsystems} considers $\frakG$-invariant linear systems with the goal of studying their descent properties to the quotient. 


In \cite{Helminck_skeletalfiltrations} Helminck studies the fundamental group of a metrized curve complex in the sense of Amini and Baker \cite{AminiBaker} (which are also crucially used in \cite{ABBRI, ABBRII}). In his framework he proves a result that amounts to identifying the fundamental group of a metrized curve complex with the \'etale fundamental group of the generic fiber of its smoothening.  Theorem \ref{mainthm_realizability} could have been proved using this framework, but we decided to use the moduli-theoretic approach of \cite{CavalieriMarkwigRanganathan_tropadmissiblecovers} via $\frakG$-admissible covers in the sense of \cite{AbramovichCortiVistoli}.

Helminck's result provides a new perspective on an older result of Sa\"{\i}di \cite{Saidi_graphofgroups}, which identifies the \'etale fundamental group of the generic fiber with the profinite completion of the fundamental group of a suitable graph of groups (in the sense of Bass and Serre \cite{Serre_trees, Bass}) that encodes the fundamental group of a metrized curve complex. From a moduli-theoretic perspective, a similar observation seems to be inherent in both \cite{BertinRomagny} and \cite{Ekedahl_Hurwitzboundary}.

From a moduli-theoretic perspective, studying degenerations of $\frakG$-covers of algebraic curves is equivalent to studying the compactification of the moduli space of $\frakG$-covers in terms of the moduli space of $\frakG$-admissible covers, as constructed in \cite{AbramovichCortiVistoli} and \cite{BertinRomagny}. In \cite[Section 7]{BertinRomagny} the authors have already introduced a graph-theoretic gadget to understand the boundary strata of this moduli space: so-called \emph{modular graphs} with an action of a finite (not necessarily abelian) group $\frakG$.  

This idea seems to have appeared independently in other works as well: Chiodo and Farkas \cite{ChiodoFarkas} study the boundary of the moduli space of level curves, which is equivalent to a component of the moduli space of $\frakG$-admissible covers for a cyclic group $\frakG$, and look at cyclic covers of an arbitrary graph. Their work has been extended to an arbitrary finite group $\frakG$ by Galeotti in \cite{2019GaleottiB, 2019GaleottiA}. Finally, in \cite{SchmittvanZelm}, Schmitt and van Zelm apply a graph-theoretic approach to the boundary of the moduli space of $\frakG$-admissible covers (for an arbitrary finite group $\frakG$) to study their pushforward classes in the tautological ring of $\overline{\mathcal{M}}_{g,n}$.

In \cite{CavalieriMarkwigRanganathan_tropadmissiblecovers} Cavalieri, Markwig, and Ranganathan develop a moduli-theoretic approach to the tropicalization of the moduli space of admissible covers (without a fixed group operation). 
 In \cite{CaporasoMeloPacini}, Caporaso, Melo, and Pacini study the tropicalization of the moduli space of spin curves, which, in view of the results in \cite{2018JensenLen}, is closely related to our story in the case $\frakG=\Z/2\Z$. 

The problem of classifying covers of a graph with an action of a given group (not necessarily abelian) was studied by Corry in \cite{2011Corry,2012Corry,2015Corry}. However, Corry considered a different category of graph morphisms, allowing edge contraction but not dilation. To the best of our knowledge, no author has considered the problem of classifying all covers of a given graph with an action of a fixed group.

\subsection*{Acknowledgements}
The authors would like to thank Matthew Baker, Madeline Brandt, Renzo Cavalieri, Gavril Farkas, Paul Helminck, David Jensen, Andrew Obus, Sam Payne, Dhruv Ranganathan, Felix R\"ohrle, Matthew Satriano, Johannes Schmitt, Pedro Souza, and Jason van Zelm for useful discussions.  In addition, the authors thank the anonymous referee(s) for their helpful comments and suggestions. 

This project has received funding from the European Union's Horizon 2020 research and innovation programme  under the Marie-Sk\l odowska-Curie Grant Agreement No. 793039. \includegraphics[height=1.7ex]{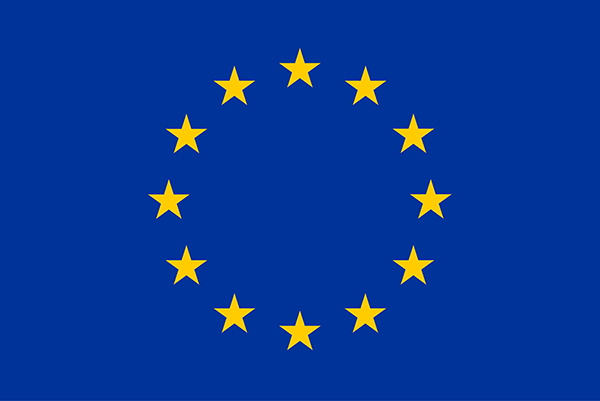}
M.U.\ has received funding from the Deutsche Forschungsgemeinschaft  (DFG, German Research Foundation) TRR 326 \emph{Geometry and Arithmetic of Uniformized Structures}, project number 444845124, by the Deutsche Forschungsgemeinschaft (DFG, German Research Foundation) Sachbeihilfe \emph{From Riemann surfaces to tropical curves (and back again)}, project number 456557832, and from the LOEWE grant \emph{Uniformized Structures in Algebra and Geometry}.  Y.L. has received funding from the EPSRC New Investigator Award (grant number EP/X002004/1).



\section{Harmonic covers of metric graphs and tropical curves}

\label{sec:definitions}

In this section, we recall a number of standard definitions concerning graphs, tropical curves, harmonic morphisms, and group actions on graphs.

\subsection{Finite graphs and harmonic morphisms} We use a modified version of Serre's definition of a graph~\cite{Serre_trees} that allows for legs, which are a type of extremal edge with no end vertex. 

\begin{definition} A \emph{graph with legs} $G$, or simply a \emph{graph}, consists of the following data:
\begin{enumerate}
\item A finite set $X(G)$.

\item An idempotent \emph{root map} $r:X(G)\to X(G)$.

\item An involution $\iota:X(G)\to X(G)$ whose fixed set contains the image of $r$.

\end{enumerate}

\end{definition}

The set $X(G)$ is the union of the \emph{vertices} $V(G)$ and \emph{half-edges} $H(G)$ of the graph $G$, where $V(G)$ is the image of $r$ and $H(G)=X(G)\backslash V(G)$ is the complement. The involution $\iota$ preserves $H(G)$ and partitions it into orbits of sizes 1 and 2; we call these respectively the \emph{legs} and \emph{edges} of $G$ and denote the corresponding sets by $L(G)$ and $E(G)$. The root map assigns one root vertex to each leg and two root vertices to each edge (each vertex is rooted at itself). A \emph{loop} is an edge whose root vertices coincide. An \emph{orientation} on $G$ is a choice of order $(h,h')$ on each edge $e=\{h,h'\}$ of $G$ and defines \emph{source} and \emph{target} maps $s,t:E(G)\to V(G)$ by $s(e)=r(h)$ and $t(e)=r(h')$. We note that a leg does not have a vertex at its free end and is thus distinct from an extremal edge, and that legs do not require orienting.

Graphs with legs naturally appear in tropical moduli problems, where a leg represents  the tropicalization of a marked point. An extremal edge, on the other hand, represents an irreducible component attached to the rest of the curve at a single node.

The \emph{tangent space} $T_v G$ and \emph{valency} $\val (v)$ of a vertex $v\in V(G)$ are defined by 
\[
T_vG=\big\{h\in H(G):r(h)=v\big\} \quad\textrm{ and } \quad\val(v)=|T_v G|,
\]
so that a leg is counted once for valency, while a loop is counted twice.

A \emph{morphism} of graphs $f:G'\to G$, is a set map $f:X(G')\to X(G)$ that commutes with the root and involution maps and that sends vertices to vertices, edges to edges, and legs to legs. By abuse of notation, we denote by $f$ the corresponding maps on the vertices, half-edges, edges, and legs. We note that our graph morphisms are \emph{finite} and do not allow edges or legs to contract to vertices. Non-finite morphisms are relevant to tropical geometry, but do not occur as quotients by finite group actions;  so we do not consider them.

Let $G$ and $G'$ be graphs. A \emph{harmonic morphism} $(f,d_f)$ consists of a graph morphism $f:G'\to G$ and a \emph{degree} assignment $d_f:X(G')\to \ZZ_{>0}$ such that $d_f(h'_1)=d_f(h'_2)$ for each edge $e'=\{h'_1,h'_2\}\in E(G')$ (a quantity that we denote by $d_f(e')$), and such that 
\begin{equation}\label{eq:localDegree}
d_{f}(v')=\sum_{h'\in T_{v'} G'\cap f^{-1}(h)}d_{f}(h')
\end{equation}
for every $v'\in V(G')$ and every $h\in T_{f(v)}G$. In particular, the quantity appearing on the right hand side of \eqref{eq:localDegree} does not depend on the choice of $h\in T_{f(v)}G$. 
The degree $d_f$ is also called the \emph{dilation factor} of $f$. If $G$ is connected, then the \emph{global degree} of $f$ is defined as
\[
\deg(f)=\sum_{v'\in f^{-1}(v)}d_f(v')=\sum_{e'\in f^{-1}(e)}d_f(e')=\sum_{l'\in f^{-1}(l)}d_f(l')
\]
for any choice of $v\in V(G)$, $e\in E(G)$ or $l\in L(G)$. 
A harmonic morphism $(f,d_f)$ is called \emph{free} if $d_f(x)=1$ for all $x\in X(G)$; a free harmonic morphism is a covering space in the topological sense.

\subsection{Group quotients and harmonic Galois covers} 
An \emph{automorphism}  of a graph $G$ is a morphism $f:G\to G$ that has an inverse; such a morphism can be made harmonic by setting $d_f=1$ everywhere. A priori, a non-trivial automorphism may \emph{flip edges}, in other words exchange the two half-edges making up an edge. Such  automorphisms do not give rise to a quotient, however, since we do not allow an edge to map to a leg. Hence we exclude them from consideration.

\begin{definition} Let $G$ be a graph and  $\frakG$  a finite group. A \emph{$\frakG$-action} on $G$ is a homomorphism from $\frakG$ to the automorphism group $\Aut (G)$, such that, for every $g\in \frakG$ and every $e=\{h,h'\}\in E(G)$, we have $g(h)\neq h'$ (so that either $g(h)=h$ and $g(h')=h'$, or $g(e)\neq e$). 
    
\end{definition}

Given a $\frakG$-action on a graph $G$, we can naturally form the quotient graph $G/\frakG$ in such a way that the quotient map $f:G\to G/\frakG$ is harmonic of degree $|\frakG|$.

\begin{definition} Let $G$ be a graph and let $\frakG$ be a finite group. Given a $\frakG$-action on $G$, we define the \emph{quotient graph} $G/\frakG$ by setting $X(G/\frakG)=X(G)/\frakG$. The root and involution maps on $G$ are $\frakG$-invariant and descend to $X(G/\frakG)$. It is clear that $V(G/\frakG)=V(G)/\frakG$ and $H(G/\frakG)=H(G)/\frakG$, and by assumption the $\frakG$-action does not identify the two half-edges of any edge of $G$. Therefore $E(G/\frakG)=E(G)/\frakG$ and $L(G/\frakG)=L(G)/\frakG$, and the quotient map 
\[
f:G\longrightarrow G/\frakG
\]
is a finite morphism. By the orbit-stabilizer theorem, we can promote $f$ to a harmonic morphism of global degree $\deg(f)=|\frakG|$ by setting $d_f(x)=|\frakG_x|$, where $\frakG_x$ is the stabilizer subgroup of $x\in X(G)$.

\label{def:Gquotient}

\end{definition}

We now define a harmonic Galois cover of a graph to be any harmonic morphism obtained in this way.

\begin{definition} Let $G$ be a graph and let $\frakG$ be a finite group of order $d$. A \emph{harmonic $\frakG$-cover} of $G$ is a harmonic morphism $f\colon G'\rightarrow G$ of degree $d$ together with a $\frakG$-action on $G'$ such that following axioms hold:
\begin{enumerate}[(i)]
\item The harmonic morphism $f$ is $\frakG$-invariant, in other words $f(g(x'))=f(x')$ and $d_f(g(x'))=d_f(x')$ for all $x'\in X(G')$ and all $g\in \frakG$.
\item For all $x\in X(G)$, the group $\frakG$ acts transitively on the fiber $f^{-1}(x)$.
\end{enumerate}
\end{definition}

Let $f:G'\to G$ be a harmonic $\frakG$-cover, and pick a vertex or half-edge $x\in X(G)$. The group $\frakG$ acts transitively on the fiber $f^{-1}(x)$, so we can identify the latter with $\frakG/\frakG_{x'}$, where $\frakG_{x'}$ is the stabilizer of some $x'\in f^{-1}(x)$. On the other hand, for any $x',x''\in f^{-1}(x)$ we have $|\frakG_{x'}|=|\frakG_{x''}|$ and $d_f(x')=d_f(x'')$. Since the degrees of $f$ on the fiber $f^{-1}(x)$ add up to $\deg f=|\frakG|$, it follows that $d_f(x')=|\frakG_{x'}|$ for any $x'\in X(G')$. It follows that $f$ is the quotient morphism $G'\to G'/\frakG$.

\subsection{Metric graphs}

Let $G$ be a graph and let $\ell:E(G)\to \RR_{>0}$ be an assignment of positive real lengths to the edges of $G$. The pair $(G,\ell)$, known as a \emph{model} for $G$, determines a \emph{metric graph} $\Ga$ by gluing a closed line segment $[0,\ell(e)]$ for each edge $e\in E(G)$ and an open infinite interval $[0,\infty)$ for each leg $l\in L(G)$ in accordance with the structure of $G$. We equip $\Ga$ with the shortest-path metric. We note that the set of legs does not depend on the choice of model and that the metric graph $\Ga$ is compact if and only if $G$ has no legs. 

A model $(G,\ell)$ of a metric graph $\Ga$ is called \emph{simple} if $G$ has no loops or multi-edges. Given a simple model $(G,\ell)$ for $\Ga$, we define the \emph{star cover}
\[
\calU(G)=\{U_v\}_{v\in V(G)}\cup \{U_l\}_{l\in L(G)}
\]
of $\Ga$ as follows. For each leg $l\in L(G)$, let $U_l\subset \Ga$ be the interior of the corresponding infinite segment in $\Gamma$.  For each vertex $v\in V(G)$, let $U_v\subset \Ga$ be the union of $v$ and the interiors of all legs and edges incident to $v$. The distinct $U_l$ have empty intersections, and $U_l\cap U_v=U_l$ if $l$ is rooted at $v$ and is empty otherwise. Finally, for distinct vertices $v$ and $w$, the intersection $U_v\cap U_w$ is either the open edge connecting $v$ and $w$, if there is such an edge, or is empty otherwise. Hence each element of $\calU(G)$ is contractible, pairwise intersections are open intervals or empty, and all triple intersections are empty, making the star cover convenient for cohomological calculations.




We now define harmonic morphisms and Galois covers of metric graphs. Let $\Ga$ and $\Ga'$ be metric graphs with models $G$ and $G'$, respectively. Let $f:G'\to G$ be a harmonic morphism of graphs satisfying the condition
\begin{equation}
    \label{eq:lengths}
\ell\big(f(e')\big)=d_f(e')\ell(e')
\end{equation}
for all $e'\in E(G')$. We define an associated continuous map $\phi:\Ga'\to \Ga$ of metric graphs by mapping vertices to vertices, edges to edges, and legs to legs according to $f$. Along each edge and leg of $\Ga'$, the map $\phi$ is linear with positive integer slope, or \emph{dilation factor}, given by the degree $d_f$ (which we also denote $d_{\phi}$). Condition~\eqref{eq:lengths} ensures that $\phi$ is continuous, and no condition is required along the infinite legs.

A \emph{harmonic morphism of metric graphs} $\phi:\Ga'\to \Ga$ is any continuous, piecewise-linear map obtained in this manner, with nonzero integer slopes given by the degree function $d_f$ of a harmonic morphism of graphs $f:G'\to G$ (and thus satisfying the balancing condition~\eqref{eq:localDegree} at each vertex of $\Ga'$). This definition is equivalent to requiring that $\phi$ pulls back harmonic functions on $\Ga$ to harmonic functions on $\Ga'$. We refer to the datum $(G,G',f:G'\to G,d_f)$ as a \emph{model} for $\phi$. We say that $\phi$ is \emph{free} if $f$ is free, or equivalently, if $\phi$ is a covering isometry.

We similarly define harmonic Galois covers of metric graphs. 
 
\begin{definition} Let $\Ga$ be a metric graph and let $\frakG$ be a finite group of order $d$. A \emph{harmonic $\frakG$-cover} of $\Ga$ is a harmonic morphism $\phi\colon \Gamma'\rightarrow \Gamma$ of degree $d$
together with an operation of $\frakG$ on $\Gamma'$ by invertible isometries such that following properties hold:
\begin{enumerate}[(i)]
\item The harmonic cover $\phi$ is $\frakG$-invariant, i.e.~$\phi(g(p'))=\phi(p')$ for all $p'\in \Ga'$ and all $g\in \frakG$.
\item For all $p\in\Gamma$, the group $\frakG$ operates transitively on the fiber $\phi^{-1}(p)$.
\end{enumerate}
\end{definition}

It is clear that a harmonic $\frakG$-cover $\phi:\Ga'\to\Ga$ of metric graphs admits a model $f:G'\to G$ that is a harmonic $\frakG$-cover of finite graphs (the models $G'$ and $G$ need to be sufficiently fine to avoid edge-flipping). For any $p'\in \Ga'$, the degree $d_\phi(p')$ is equal to the order of the stabilizer group $\frakG_{p'}$.



\subsection{Weighted graphs, tropical curves and ramification} \label{subsec:ramification} Graphs and metric graphs that arise as tropicalizations of algebraic curves come equipped with an additional vertex weight function that records local genera. These weights allow us to capture the auxiliary phenomenon of \emph{ramification} for harmonic morphisms. We recall the definitions.

A \emph{weighted graph} is a pair $(G,g)$, where $G$ is a finite graph and $g:V(G)\to \ZZ_{\geq 0}$ is a function, where $g(v)$ is called the \emph{genus} of the vertex $v$. Similarly, a \emph{tropical curve} $(\Ga,g)$ is a metric graph $\Ga$ together with a function $g:\Ga\to \ZZ_{\geq 0}$ with finite support. When choosing a model $(G,\ell)$ for a tropical curve $(\Ga,g)$, we assume that each point $x\in \Ga$ with $g(x)>0$ corresponds to a vertex, and not to an interior point of an edge or a leg, so that $(G,g)$ is a weighted graph. A \emph{harmonic morphism} of tropical curves is a harmonic map of the underlying metric graphs. 

Let $(G,g)$ be a weighted graph. We define the \emph{Euler characteristic} $\chi(v)$ of a vertex $v\in V(G)$ as
\[
\chi(v)=2-2g(v)-\val(v). 
\]
Now let $f:G'\to G$ be a harmonic morphism of weighted graphs $(G',g')$ and $(G,g)$. We define the \emph{ramification degree} of $f$ at a vertex $v\in V(G')$ to be the quantity
\begin{equation}
\Ram_f(v')=d_f(v')\chi(f(v'))-\chi(v').
\label{eq:ramification}
\end{equation}
We say that $f$ is \emph{unramified} if it satisfies the \emph{local Riemann--Hurwitz condition} $\Ram_f(v')=0$ for all $v'\in V(G')$, where we note that, in contrast to the algebraic setting, it is possible for the ramification degree at a vertex to be negative. A harmonic morphism $\phi:\Ga'\to \Ga$ of tropical curves is \emph{unramified} if it has an unramified model. Our definition of ramification was introduced in~\cite{ulirsch2019tropical}, and is equivalent to the standard definition found in~\cite{ABBRI} or~\cite{CavalieriMarkwigRanganathan_tropadmissiblecovers}.




\section{Dilated cohomology and finite harmonic abelian covers}

In this section, we give a cohomological classification of harmonic covers of a given metric graph with abelian structure group. For the remainder of this section, we fix a finite abelian group $\frakA$.

Let $\phi\colon \Gamma'\rightarrow \Gamma$ be a harmonic $\frakA$-cover. For any point $p\in \Ga$, the stabilizer subgroups of any two points in the fiber $\phi^{-1}(p)$ are conjugate and hence equal. Therefore this group depends only on $p$, and we denote it by $D(p)\subseteq \frakA$ and call it the \emph{dilation group} of $p$. Similarly, choosing a finite graph model $f:G'\to G$ of $\phi$, we denote by $D(x)\subseteq \frakA$ the stabilizer of any element of $f^{-1}(x)$. The groups $D(x)$ fulfil the semicontinuity property $D(h)\subseteq D(v)$ for any half-edge $h\in H(G)$ rooted at a vertex $v\in V(G)$. Furthermore, for any edge $e=\{h,h'\}\in E(G)$ we have $D(h)=D(h')$, and we denote this group by $D(e)$. 



This motivates the following definition.

\begin{definition}
An \emph{$\frakA$-dilation datum} $D$ on a finite graph $G$ is a choice of a subgroup $D(v)\subseteq \frakA$ for every $v\in V(G)$ and $D(h)\subseteq \frakA$ for every $h\in H(G)$, such that $D(h)\subset D(v)$ if $h$ is rooted at $v$ and such that $D(h)=D(h')=D(e)$ for any edge $e=\{h,h'\}\in E(G)$. 
We note that if $e\in E(G)$ is an edge with root vertices $u,v\in V(G)$, then $D(e)\subseteq D(u)\cap D(v)$. An \emph{$\frakA$-dilation datum} $D$ on a metric graph $\Ga$ is an $\frakA$-dilation datum on some model $G$ of $\Ga$, which defines a subgroup $D(p)\in \frakA$ for each $p\in \Ga$.
\end{definition}



An $\frakA$-dilation datum on a metric graph $\Ga$ together with a choice of simple model naturally gives rise to a dual sheaf of abelian groups.  %

\begin{definition} Let $D$ be an $\frakA$-dilation datum on a simple model $G$ of a metric graph $\Ga$. We define the \emph{codilation sheaf} $\frakA_D$ on $\Ga$ as follows. For a vertex $v\in V(G)$, we denote $C(v)=D(v)$. Similarly, for a leg $l\in L(G)$ we denote $C(l)=D(v)$, where $v=r(l)$. Finally, for an edge $e\in E(G)$ with root vertices $v$ and $w$, we denote $C(e)=D(v)+D(w)\subseteq \frakA$. We note that $D(e)\subseteq C(e)$ for any edge $e\in E(G)$ and $D(l)\subseteq C(l)$ for any leg $l\in L(G)$. Now let $\calU(G)=\{U_v,U_l\}$ be the star cover of $\Ga$ associated to $G$.  The sections of $\frakA_D$ over the open cover and the induced intersections are
\begin{equation*}
\frakA_D(U_v)= A/C(v),\quad
\frakA_D(U_l)= A/C(l),\quad \textrm{and }\quad
\frakA_D(U_e)= A/C(e), 
\end{equation*}
where $U_e=U_v\cap U_w$ if $e$ is  the (unique) edge between $v$ and $w$. The restriction maps are induced by the inclusions $D(v)=C(v)\subseteq C(e)$ and $D(v)=C(v)=C(l)$ for an edge $e$ or a leg $l$ rooted at a vertex $v$. Given a connected open set $U\subseteq \Ga$, we set $\frakA_D(U)=\frakA_D(U_v)$ if $v\in U\subseteq U_v$ for some vertex $v$, while  $\frakA_D(U)=\frakA_D(U_e)$ and $\frakA_D(U)=\frakA_D(U_l)$ respectively
if $U\subseteq U_e$ or $U\subseteq U_l$. For larger open sets, we define the space of sections via the sheaf axioms. 

The \emph{dilated cohomology group} of the pair $(\Ga,D)$ is the sheaf cohomology group $H^1(\Ga,\frakA_D)$. We note that the sheaf $\frakA_D$ depends on the choice of model (see Example~\ref{ex:codilation} below), but the group $H^1(\Ga,\frakA_D)$ does not.

\end{definition}

We now show that harmonic $\frakA$-covers of $\Ga$ are in natural bijection with $\frakA_D$-torsors. We first recall the definition of torsors over a sheaf of abelian groups, and their description in terms of \v{C}ech cocycles. Let $\calF$ be a sheaf of abelian groups on a topological space $X$. We may view $\calF$ as a sheaf of $\calF$-sets, with each group acting on itself by translation. An \emph{$\calF$-torsor} $\calT$ on $X$ is a locally trivial sheaf of $\calF$-sets, in other words a sheaf of $\calF$-sets such that $X$ admits a cover by open sets $U$ with the property that $\calT|_U$ and $\calF|_U$ are isomorphic as sheaves of $\calF$-sets. 


It is well-known that the set of isomorphism classes of $\calF$-torsors on $X$ is the sheaf cohomology group $H^1(X,\calF)$. We explicitly calculate this group for a codilation sheaf $\frakA_D$ on a metric graph $\Ga$ as a \v{C}ech cohomology group. Choose an oriented simple model $G$ for $\Ga$, then the star cover $\calU(G)=\{U_v,U_l\}$ is acyclic for $\frakA_D$. Let $\calT$ be an $\frakA_D$-torsor, then we can find trivializations $g_v:\calT|_{U_v}\to \frakA_D|_{U_v}$. Each edge $e\in E(G)$ corresponds to a nonempty intersection $U_e=U_{s(e)}\cap U_{t(e)}$, and the composed isomorphism $g^e=g_{t(e)}|_{U_e}\circ (g_{s(e)}|_{U_e})^{-1}:\frakA|_{U_e}\to \frakA|_{U_e}$ is given by translation by an element of $\frakA(U_e)=A/C(e)$, which we also denote by $g^e$. Hence the $\frakA_D$-torsor $\calT$ determines a tuple $(g^e)_{e\in E(G)}$, where $g^e\in A/C(e)$. Choosing different trivializations for $\calT$ over the sets $U_v$ determines a different tuple $(\tilde{g}^e)$, and composing the trivializations produces elements $g^v\in \frakA_D(U_v)=A/C(v)$ for $v\in V(G)$ such that $\tilde{g}^e-g^e=g^{t(e)}-g^{s(e)}$ in the common quotient group $\frakA/C(e)$. All triple intersections are empty, so the cocycle condition is trivially verified and the tuple $(g^e)$ determines an element of $\check{H}^1\big(\calU(G),\frakA_D\big)\cong H^1(\Ga,\frakA_D)$, and we can reverse the construction to obtain $\calT$ from $(g^e)$. 

We now state our main result, which shows that harmonic $\frakA$-covers with fixed $\frakA$-dilation datum $D$ are classified by the dilated cohomology group $H^1(\Ga,\frakA_D)$.


\begin{theorem}\label{prop_torsor=harmonicGcover} Let $\Ga$ be a metric graph and let $D$ be an $\frakA$-dilation datum on $\Ga$. There is a natural one-to-one correspondence between $\frakA_D$-torsors on $\Gamma$ and harmonic $\frakA$-covers of $\Gamma$ with associated $\frakA$-dilation datum $D$. 
\end{theorem}

\begin{proof} Choose an oriented simple model $G$ for $\Ga$ such that $D$ is defined over $G$. Let $\phi\colon \Gamma'\rightarrow \Gamma$ be a harmonic $\frakA$-cover with $\frakA$-dilation datum $D$ and let $f:G'\to G$ be a model for $\phi$. For any vertex $v\in V(G)$, the fiber $f^{-1}(v)$ is naturally a torsor over $\frakA_D(U_v)=\frakA/D(v)$. The fiber $f^{-1}(e)$ over an edge $e$, however, is a torsor over $\frakA/D(e)$, not over $\frakA_D(U_e)=\frakA/C(e)$. The latter group is a quotient of the former, and we replace $f^{-1}(e)$ by its quotient by $C(e)/D(e)$. Similarly, for each leg $l\in L(G)$ we take the quotient of $f^{-1}(l)$ by $C(l)/D(l)$. In this way, we obtain an $\frakA_D$-torsor on $\Ga$. We observe that, generally speaking, the espace \'etal\'e of this torsor is not Hausdorff, since if $D(v)\subsetneq C(e)$ then the vertex $v$ has more preimages than the adjacent edge $e$.


Conversely, let $\calT$ be an $\frakA_D$-torsor over $\Ga$. We construct a harmonic $\frakA$-cover $f:G'\to G$ by resolving the espace \'etal\'e of $\calT$ in a canonical way. Let $(g^e)\in \check{H}^1(\calU(G),\frakA_D)$ be a \v{C}ech cocycle representing $\calT$. We arbitrarily lift each $g^e\in A/C(e)$ to an element $\widetilde{g}^e\in\frakA/D(e)$. For each vertex $v\in V(G)$, the fiber $f^{-1}(v)$ is equal to $\frakA/D(v)$ as an $\frakA$-set. For each edge $e\in V(G)$ with source and target vertices $v=s(e)$ and $w=t(e)$, the fiber $f^{-1}(e)$ is $\frakA/D(e)$. The gluing map $f^{-1}(e)\to f^{-1}(v)$ is the natural quotient map $\frakA/D(e)\to \frakA/D(v)$, while the gluing map $f^{-1}(e)\to f^{-1}(w)$ is translation by $\widetilde{g}^e$ followed by taking the quotient. Finally, for each leg $l\in L(G)$ with root vertex $v$, we set $f^{-1}(l)=A/D(l)$, and the root map $f^{-1}(l)\to f^{-1}(v)$ is the quotient map $\frakA/D(l)\to \frakA/D(v)$.

One may now verify that these constructions are inverses of each other, thereby completing the proof.

\end{proof}

\begin{example} \label{ex:codilation}
In the following picture, on the left, we illustrate a harmonic $\Z/4\Z$-cover, for which the $\Z/4\Z$-dilation datum is given by $D(u)=\Z/2\Z$, $D(v)=\Z/4\Z$, and $D(e)=0$. In this case we have $C(e)=\Z/4\Z$ and thus $(\frakA_D)_u=\Z/2\Z$, $(\frakA_D)_v=0$ , and $(\frakA_D)_e=0$. The (non-Hausdorff) espace \'etal\'e of the associated $\calA_D$-torsor is illustrated on the right.  

\begin{center}\begin{tikzpicture}
\path[draw] (0,0) -- (2,0);
\fill[black] (0,0) circle (0.1);
\fill[black] (2,0) circle (0.1);

\fill[black] (0,1) circle (0.1);
\fill[black] (0,3) circle (0.1);
\fill[black] (2,2) circle (0.1);

\draw (0, 1) .. controls (1.25,1.25) .. (2, 2);
\draw (0, 1) .. controls (0.75,1.75) .. (2, 2);
\draw (0, 3) .. controls (1.25,2.75) .. (2, 2);
\draw (0, 3) .. controls (0.75,2.25) .. (2, 2);

\draw[->] (1,1) -- (1,0.25);

\node (u) at (-0.25,-0.25) {$u$};
\node (v) at (2.25,-0.25) {$v$};
\node (e) at (1,-0.25) {$e$};
\node (v') at (3.3,2) {$D(v)=\ZZ/4\ZZ$};
\node (u') at (-1.3,1) {$D(u)=\ZZ/2\ZZ$};
\node (u'') at (-1.3,3) {$D(u)=\ZZ/2\ZZ$};
\node (e') at (1.25,3.25) {$D(e)=0$};

\path[draw] (7,0) -- (9,0);
\fill[black] (7,0) circle (0.1);
\fill[black] (9,0) circle (0.1);

\fill[black] (7,1) circle (0.1);
\fill[black] (7,3) circle (0.1);
\draw (7,2) circle (0.1);
\fill[black] (9,2) circle (0.1);
\path[draw] (7.1,2) -- (9,2);

\draw[->] (8,1) -- (8,0.25);

\node (uu) at (6.75,-0.25) {$u$};
\node (vv) at (9.25,-0.25) {$v$};
\node (ee) at (8,-0.25) {$e$};
\node (vv') at (10,2) {$(\frakA_D)_v=0$};
\node (uu') at (5.45,1) {$(\frakA_D)_u=\ZZ/2\ZZ$};
\node (uu'') at (5.45,3) {$(\frakA_D)_u=\ZZ/2\ZZ$};
\node (ee') at (8,2.5) {$(\frakA_D)_e=0$};

\end{tikzpicture}
\end{center}
Consider now a subdivision of the base with extra vertex $w$. Then the espace \'etal\'e of the associated codilation sheaf is given as follows:
\begin{center}\begin{tikzpicture}
\path[draw] (0,-1) -- (4,-1);
\fill[black] (0,-1) circle (0.1);
\fill[black] (4,-1) circle (0.1);
\fill[black] (2,-1) circle (0.1);

\fill[black] (0,1) circle (0.1);
\fill[black] (0,3) circle (0.1);
\fill[black] (4,2) circle (0.1);
\draw (2,2) circle (0.1);
\draw[black] (2,3) circle (0.1);
\draw[black] (2,1) circle (0.1);
\fill[black] (0,1) circle (0.1);

\draw (0, 3) -- (1.9, 3);
\draw (0, 1) -- (1.9, 1);
\draw (2.1, 2) -- (4, 2);

\fill[black] (2,0.5) circle (0.1);
\fill[black] (2,1.5) circle (0.1);
\fill[black] (2,2.5) circle (0.1);
\fill[black] (2,3.5) circle (0.1);

\draw[->] (2,0.25) -- (2,-0.75);

\node (u) at (-0.25,-1.25) {$u$};
\node (v) at (4.25,-1.25) {$v$};
\node (w') at (3.5,3.5) {$(\frakA_D)_w=\Z/4\Z$};
\node (v') at (5.3,2) {$(\frakA_D)_v=0$};
\node (u'') at (-1.4,3) {$(\frakA_D)_u=\ZZ/2\ZZ$};
\node (w) at (2,-1.25) {$w$};

\end{tikzpicture}\end{center}

\end{example}



We point out that the dilated cohomology group $H^1(\Ga,\frakA_D)$ only depends on the dilation factors at vertices and not on the dilation factors along the edges. The interpretation of a class in $H^1(\Gamma,\frakA_D)$ in Proposition \ref{prop_torsor=harmonicGcover}, however, does depend on the dilation along edges. That is, different choices of dilation factors would lead to different edge lengths in the corresponding harmonic covers.

We now determine when a harmonic $\frakA$-cover $\phi:\Ga'\to \Ga$ of tropical curves is unramified. Let $f:G'\to G$ be a model of $\phi$, where $G'$ and $G$ are weighted graphs, and let $v'\in V(G')$ be a vertex lying over $v=f(v')$. The number of half-edges $h'\in T_{v'}(G')$ that are rooted at $v'$ and that lie over a given half-edge $h\in T_vG$ is equal to the order of the corresponding quotient $\big|D(v)\big/D(h)\big|$. A short calculation then shows that $\Ram_f(v')=0$ if and only if
\begin{equation}    
g(v')=1+\big|D(v)\big|\big(g(v)-1\big)+\frac{\big|D(v)\big|}{2}\sum_{h\in T_vG}\left[1-\frac{1}{|D(h)|}\right].
\label{eq:admissible}
\end{equation}
Since $g(v)$ and $g(v')$ are non-negative integers, this condition imposes certain restrictions on the $\frakA$-dilation datum of an unramified harmonic $\frakA$-cover. As an example, we consider the simplest case of a cyclic cover of prime order.

\begin{example} \label{ex:semistable} Let $\phi:\Ga'\to\Ga$ be an unramified harmonic $\frakA$-cover of tropical curves with Galois group $\frakA=\ZZ/p\ZZ$, where $p\geq 2$ is prime, and let $f:G'\to G$ be a model. For any element $x\in X(G)$ we have either $D(x)=\ZZ/p\ZZ$ or $D(x)=1$, and we say that $x$ is \emph{dilated} or \emph{undilated}, respectively. The set of dilated vertices and half-edges forms the \emph{dilation subgraph} $G_{\dil}\subseteq G$. 

Now let $v'\in V(G')$ be a vertex mapping to $v=f(v')$. If $v$ is undilated, Equation~\ref{eq:admissible} simply reads $g(v')=g(v)$. For a dilated vertex $v\in V(G_{\dil})$, let $d(v)=\big|\{h\in T_vG\mid D(h)=\ZZ/p\ZZ\}\big|$ be the valency of $v$ in $G_{\dil}$. Equation~\ref{eq:admissible} then imposes the following conditions on $g(v)$ and $d(v)$:

\begin{itemize}
    \item If $p=2$, then $d(v)\geq 2$ or $g(v)\geq 1$, and in addition $d(v)$ is even.
    \item If $p\geq 3$, then $d(v)\geq 2$ or $g(v)\geq 1$.
\end{itemize}

In other words, the dilation subgraph $G_{\dil}$ is \emph{semistable}, and additionally if $p=2$ then each vertex of $G_{\dil}$ has even valency (see Lemma 5.4 in~\cite{2018JensenLen}).

\end{example}

\section{Moduli of admissible $\frakG$-covers and their tropicalization} \label{sec:tropicalization}

Let $\mathfrak{\frakG}$ be a fixed finite group, which, in this section, does not need to be abelian. In the following, we explain how harmonic $\frakG$-covers of weighted graphs and tropical curves naturally arise as tropicalizations of algebraic $\frakG$-covers from a moduli-theoretic perspective, expanding on~\cite{ACP} and~\cite{CavalieriMarkwigRanganathan_tropadmissiblecovers} (recall that unramified harmonic morphisms of tropical curves are called \emph{tropical admissible covers} in~\cite{CavalieriMarkwigRanganathan_tropadmissiblecovers}). We always work over $\Spec \Z\big[\frac{1}{\vert \frakG\vert}\big]$ to avoid the wild world of non-tame covers. 

\subsection{Compactifying the moduli space of $\frakG$-covers}
Let $X\rightarrow S$ be a family of smooth projective curves of genus $g\geq 2$ with $n$ marked disjoint sections $s_1,\ldots, s_n\in X(S)$.  A $\frakG$-\emph{cover} of $X$ is a finite morphism $X'\rightarrow X$ together with an operation of $\frakG$ on $X'$ over $X$ that is a principal $\frakG$-bundle on the complement of the sections, as well as a  marking $s'_{ij}\in X'(S)$ of the disjoint preimages of the $s_i$, indexed by $i=1,\ldots, n$ and $j=1,\ldots, k_i$. Denote by $\calH_{g,\frakG}$ the moduli space of connected $\frakG$-covers of smooth curves of genus $g$ (see e.g. \cite{RomagnyWewers} for a construction). There is a good notion of a limit object as $X$ degenerates to a stable curve, as introduced in \cite{AbramovichCortiVistoli}. 

\begin{definition}\label{def_admissibleGcover} Let $\frakG$ be a finite group and let $X\rightarrow S$ be a family of stable curves of genus $g\geq 0$ with $n$ marked disjoint sections $s_1,\ldots, s_n$.  Let $\mu=(r_1,\ldots, r_n)$ be an $n$-tuple of natural numbers that divide $|\frakG|$, and denote $k_i=\vert\frakG\vert/r_i$ for $i=1,\ldots, n$. An \emph{admissible $\frakG$-cover} of $X$ consists of a finite morphism $X'\rightarrow X$ from a family of stable curves $X'\rightarrow S$, 
an action of $\frakG$ on $X'$, and disjoint sections $s'_{ij}$ of $X'$ over $S$ for $i=1,\ldots,n$ and $j=1,\ldots,k_i$, 
subject to the following conditions: 
\begin{enumerate}[(i)]
\item The morphism $X'\rightarrow X$ is a principal $\frakG$-bundle away from the nodes and sections of $X$.
\item The preimage of the set of nodes in $X$ is precisely the set of nodes of $X'$.
\item The preimage of a section $s_i$ is precisely given by the sections $s'_{i1},\ldots, s'_{ik_i}$.
\item Let $p$ be a node in $X$ and $p'$ a node of $X'$ above $p$. Then $p'$ is \'etale-locally given by $x'y'=t$ for a  suitable $t\in\calO_S$ and $p$ is \'etale-locally given by $xy=t^r$ for some integer $r\geq 1$ with $(x')^r=x$ and $(y')^r=y$, and the stabilizer of $\frakG$ at $p'$ is cyclic of order $r$ and operates via 
\begin{equation*}
(x',y')\longmapsto (\zeta x',\zeta^{-1} y')
\end{equation*}
for an $r$-th root of unity $\zeta\in\mu_r$. 
\item \'Etale-locally near the sections $s_i$ and $s'_{ij}$, the morphism $X'\rightarrow X$ is given by $\calO_S[t_i]\rightarrow \calO_S[t_{ij}']$ with $(t_{ij}')^{r_i}=t_i$ for appropriate choices of $t_i$ and $t'_{ij}$, and the stabilizer of $\frakG$ along $s_{ij}$ is cyclic of order $r_i$ and operates via $t'_{ij}\mapsto \zeta t'_{ij}$, for an $r_i$-th root of unity $\zeta\in \mu_{r_i}$.
\end{enumerate}
\end{definition}


We emphasize that the $\frakG$-action is part of the data; so, in particular, an isomorphism between two admissible $\frakG$-covers has to be a $\frakG$-equivariant isomorphism. As explained in \cite{AbramovichCortiVistoli}, the moduli space $\calHbar_{g,\frakG}(\mu)$ of admissible $\frakG$-covers of stable $n$-marked curves of genus $g$ is a smooth and proper Deligne--Mumford stack over $\Spec \Z[\frac{1}{\vert \frakG\vert}]$ that contains the locus $\calH_{g,\frakG}(\mu)$ of $\frakG$-covers of smooth curves of ramification type $\mu$  as an open substack. The complement of $\calH_{g,\frakG}(\mu)$ is a normal crossing divisor. 

\begin{remark}
Although closely related, the moduli space $\calHbar_{g, \frakG}(\mu)$ is actually not quite the same as the one constructed in \cite{AbramovichCortiVistoli}. The quotient 
\begin{equation*}
\big[\calHbar_{g, \frakG}(\mu)/S_{k_1}\times \ldots \times S_{k_{n}}\big]
\end{equation*}
which forgets about the order of the marked sections on $s'_{ij}$ of $X'$ over $S$ for $i=1,\ldots,n$ and $j=1,\ldots,k_i$, is equivalent to a connected component of the moduli space of twisted stable maps to $\mathbf{B}\frakG$ in the sense of \cite{AbramovichVistoli, AbramovichCortiVistoli}, indexed by ramification profile and decomposition into connected components. Our variant of this moduli space $\calHbar_{g, \frakG}(\mu)$, with ordered sections on $X'$, has also appeared in  \cite{SchmittvanZelm} and in \cite{JarvisKaufmannKimura} (the latter permitting admissible covers with possibly disconnected domains).

An object in $\calHbar_{g, \frakG}(\mu)$ is technically not an admissible $\frakG$-cover $X'\rightarrow X$ but rather a $\frakG$-cover $X'\rightarrow \calX$ of a twisted stable curve $\calX$. A \emph{twisted stable curve} $\calX\rightarrow S$ is a Deligne--Mumford stack $\calX$ with sections $s_1,\ldots, s_n\colon S\rightarrow \calX$ whose coarse moduli space $X\rightarrow S$ is a family of stable curves over $S$ with $n$ marked sections (also denoted by $s_1,\ldots, s_n$) such that 
\begin{enumerate}
\item The smooth locus of $\calX$ is representable by a scheme.
\item The singularities are \'etale-locally given by $\big[\{x'y'=t\}/\mu_r\big]$ for $t\in\calO_S$, where $\zeta\in\mu_r$ acts by $\zeta\cdot(x',y')=(\zeta x',\zeta^{-1}y')$. In this case the singularity in $X'$ is locally given by $xy=t^{r}$.
\item The stack $\calX$ is a root stack $\big[\sqrt[r_i]{s_i/X}\big]$ along the section $s_i$ for all $i=1,\ldots, n$. 
\end{enumerate}

The two notions are naturally equivalent: given an admissible $\frakG$-cover $X'\rightarrow X$, the associated twisted $\frakG$-cover is given by $X'\rightarrow [X'/\frakG]$. Conversely, given a twisted $\frakG$-cover $X'\rightarrow \calX$ in the corresponding connected component, the composition $X'\rightarrow\calX\rightarrow X$ with the morphism to the coarse moduli space $X$ is an admissible $\frakG$-cover. We refer the interested reader to \cite{BertinRomagny} for an alternative construction. 
 \end{remark}

\subsection{From algebraic to tropical $\frakG$-covers}\label{section_fromalgebraictotropical} We now explain how to construct unramified harmonic $\frakG$-covers of weighted graphs and tropical curves from algebraic $\frakG$-covers.

\begin{definition}
Let $F_0:X'_0\to X_0$ be an admissible $\frakG$-cover of stable nodal curves over an algebraically closed field $k$ with $n$ smooth distinct marked points on $X_0$. The \emph{dual harmonic $\frakG$-cover} $f:G'\to G$ is defined as follows:

\begin{enumerate}
\item The graph $G$ is the dual graph of $X_0$, namely the irreducible components of $X_0$ correspond to the vertices of $G$, the nodes correspond to the edges, and the sections correspond to the legs. Similarly, $G'$ is the dual graph of $X'_0$.

\item The vertex weights $g:V(G'_0)\to \ZZ_{\geq 0}$ and $g:V(G_0)\to \ZZ_{\geq 0}$ are the genera of the normalizations of the corresponding irreducible components. 

\item The legs of $G_0$ are marked $l\colon \{1,\ldots n\}\simeq L(G_0)$ according to the full order of the marked points. 

\item The morphism $F_0:X'_0\to X_0$ sends components to components, which defines the morphism $f:V(G')\to V(G)$ on the vertices.

\item Every node $p_{e'}$ of $X'_0$ has a local equation $x'y'=0$, and maps to a node $p_e$ of $X_0$ with local equation $xy=0$ via $(x')^r=x$ and $(y')^r=y$. This defines the map on the half-edges, and $r=d_f(e')$ gives the dilation factor.

\item Let $u'_{ij}$ be a uniformizer at $s'_{ij}$ on $X'_0$. Locally near $s'_{ij}$, the morphism $F_0$ is given by $u'_{ij}=u_i^{r_i}$ for a choice of uniformizer $u_i$ at $s_i$. The dilation factor $d_f(l'_{ij})$ along the leg corresponding to $s'_{ij}$ is equal to $r_i$.

\end{enumerate}

\end{definition}

The operation of $\frakG$ on $X'_0$ induces an operation of $\frakG$ on $G'$ for which the map $f:G'\rightarrow G$ is $\frakG$-invariant. By Definition \ref{def_admissibleGcover} (iii) and (iv), the stabilizer of every edge $e'_i$ and of every leg $l'_{ij}$ is a cyclic group of order $r_i$ and $r_{ij}$, respectively.  Since $F_0:X'_0\rightarrow X_0$ is a principal $\frakG$-bundle away from the nodes, the operation of $\frakG$ on the fiber over each point in $X_0$ is transitive and so $f:G'\rightarrow G$ is a harmonic $\frakG$-cover. Applying the Riemann--Hurwitz formula to the restriction of $F_0$ to each irreducible component of $X'_0$, we observe that $f$ is unramified.

\begin{definition} Let $X$ be a smooth projective curve of genus $g$ over a non-Archimedean field $K$ (whose residue characteristic is zero or coprime to $\vert \frakG\vert$) with $n$ marked points $s_1,\ldots, s_n$ over $K$. Let $(F:X'\rightarrow X,s'_{ij})$ be a $\frakG$-cover of $X$, where $i=1,\ldots,n$ and $j=1,\ldots,k_i$. By the valuative criterion for properness, applied to the stack $\calHbar_{g,\frakG}(\mu)$, there is a finite extension $L$ of $K$ such that $X'_L\rightarrow X_L$ extends to a family of admissible $\frakG$-covers $\calF:\calX'\rightarrow\calX$ defined over the valuation ring $R$ of $L$ (with marked sections also denoted by $s_i$ and $s'_{ij}$). The \emph{dual harmonic $\frakG$-cover} $\phi:\Ga_{X'}\to \Ga_X$ is defined as follows:

\begin{enumerate} \item The graph models of the tropical curves $\Ga_{X'}$ and $\Ga_X$ are the dual graphs $G_{\calX'}$ and $G_{\calX}$ of the special fibers $\calX'_0$ and $\calX_0$, respectively. 

\item The edge length function $\ell\colon E(G_{\calX})\rightarrow\R_{>0}$ associates to an edge $e$ the positive real number $r\cdot\val(t)$, where the corresponding node of $\calX$ is \'etale-locally given by an equation $xy=t^r$ for $t\in R$. We similarly define the edge length function $\ell\colon E(G_{\calX'})\rightarrow\R_{>0}$.

\item The restriction $\calF_0:\calX'_0\to \calX_0$ of $\calF$ to the special fibers is an admissible $\frakG$-cover over $k$, and the underlying graph model for $\phi$ is the dual harmonic $\frakG$-cover $f:G_{\calX'_0}\to G_{\calX_0}$ of $\calF_0$.
    
\end{enumerate}


We note that the models $G_{\calX'}$ and $G_{\calX}$ depend on the choice of extension $\calF$, but the tropical curves $\Ga_{X'}$ and $\Ga_X$ do not. 



\end{definition}

The map $\phi\colon \Gamma_{X'}\rightarrow \Gamma_{X}$ may also seen to be harmonic by \cite[Theorem A]{ABBRI} upon identifying $\Gamma_{X'}$ and $\Gamma_X$ with the non-Archimedean skeletons of $(X')^{\an}$ and $X^{\an}$, respectively. The morphism $\phi:\Ga_{X'}\to \Ga_X$ is unramified because $f$ is unramified. 

\subsection{A modular perspective on tropicalization}
Following the recipe in \cite[Section 3.2.3]{CavalieriMarkwigRanganathan_tropadmissiblecovers} one may construct a tropical moduli space $\calH^{\trop}_{g,\frakG}(\mu)$ as a generalized cone complex that parametrizes isomorphism classes of unramified harmonic $\frakG$-covers with dilation type $\mu$ along the marked legs.

Let us now work over an algebraically closed non-Archimedean field $K$, whose residue characteristic is either zero or coprime to $\vert \frakG\vert$. Denote by $\calH_{g,\frakG}^{\an}(\mu)$ the Berkovich analytic space\footnote{We implicitly work with the underlying topological space of the Berkovich analytic stack $\calH_{g,\frakG}^{\an}(\mu)$, as introduced in \cite[Section 3]{Ulirsch_tropisquot}.} associated to $\calH_{g,\frakG}(\mu)$. The process described in Section \ref{section_fromalgebraictotropical} above defines a natural \emph{tropicalization map}
\begin{equation*}\begin{split}
\trop_{g,\frakG}(\mu)\colon \calH_{g,\frakG}^{\an}(\mu)&\longrightarrow H_{g,\frakG}^{\trop}(\mu)\\
\big[X'\rightarrow X,s_i,s'_{ij}\big]& \longmapsto \big[(\Gamma_{X'},g')\rightarrow(\Gamma_X,g)\big]
\end{split}\end{equation*}
that associates to an admissible $\frakG$-cover $X'\rightarrow X$ of smooth curves over a non-Archimedean extension $L$ of $K$ an unramified tropical $\frakG$-cover $\Gamma_{X'}\rightarrow \Gamma_X$ of the dual tropical curve $\Gamma_X$ of $X$.

Since the boundary of $\calHbar_{g,\frakG}(\mu)$ has normal crossings, the open immersion $\calH_{g,\frakG}(\mu)\hookrightarrow \calHbar_{g,\frakG}(\mu)$ is a toroidal embedding in the sense of \cite{KKMSD}. Therefore, as explained in \cite{Thuillier_toroidal, ACP, Ulirsch_nonArchTeich}, there is a natural strong deformation retraction $\rho_{g,\frakG}\colon \calH_{g,\frakG}^{\an}(\mu)\rightarrow \calH_{g,\frakG}^{\an}(\mu)$ onto a closed subset of $\calH_{g,\frakG}^{\an}(\mu)$ that carries the structure of a generalized cone complex, the \emph{non-Archimedean skeleton} $\Sigma_{g,\frakG}(\mu)$ of $\calH_{g,\frakG}^{\an}(\mu)$. Expanding on \cite[Theorem 1 and 4]{CavalieriMarkwigRanganathan_tropadmissiblecovers}, we have:

\begin{theorem}\label{thm_skeletonvstropicalization}
The tropicalization map 
$\trop_{g,\frakG}(\mu)\colon\calH_{g,\frakG}^{\an}(\mu)\longrightarrow H_{g,\frakG}^{\trop}(\mu)$
factors through the retraction to the non-Archimedean skeleton $\Sigma_{g,\frakG}(\mu)$ of $\calH_{g,\frakG}^{\an}(\mu)$, so that the restriction 
\begin{equation*}
\trop_{g,\frakG}(\mu)\colon \Sigma_{g,\frakG}(\mu)\longrightarrow H^{\trop}_{g,\frakG}(\mu)
\end{equation*}
to the skeleton is a finite strict morphism of generalized cone complexes. Moreover, the diagram
\begin{equation}\label{eq_functorialityoftrop}\begin{tikzcd}
\calH_{g,\frakG}^{\an}(\mu)\arrow[rr,"\src_{g,\frakG}^{\an}(\mu)"] \arrow[dd,"\tar_{g,\frakG}^{\an}(\mu)"']\arrow[rd,"\trop_{g,\frakG}(\mu)"]& & \calM_{g',k}^{\an}\arrow[d,"\trop_{g',k}"]\\
 & H_{g,\frakG}^{\trop}(\mu) \arrow[r,"\src_{g,\frakG}^{\trop}(\mu)"]\arrow[d,"\tar_{g,\frakG}^{\trop}(\mu)"']& M_{g',k}^{\trop}\\
\calM_{g,n}^{\an}\arrow[r,"\trop_{g,n}"']& M_{g,n}^{\trop} &
\end{tikzcd}\end{equation}
commutes.
\end{theorem}

In other words, the restriction of $\trop_{g,\frakG}(\mu)$ onto a cone in $\Sigma_{g,\frakG}(\mu)$ is an isomorphism onto a cone in $H^{\trop}_{g,\frakG}(\mu)$ and every cone in $H_{g,\frakG}^{\trop}(\mu)$ has at most finitely many preimages in $\Sigma_{g,\frakG}(\mu)$. Theorem \ref{thm_skeletonvstropicalization}, in particular, implies that the tropicalization map $\trop_{g,\frakG}(\mu)$ is well-defined, continuous, and proper.

The proof is almost word for word the same as the one of \cite[Theorems 1 and 4]{CavalieriMarkwigRanganathan_tropadmissiblecovers}. We need to observe that the construction in \cite{CavalieriMarkwigRanganathan_tropadmissiblecovers} is compatible with the $\frakG$-operation on both the algebraic and the tropical side. Moreover, using \cite[Section 4.5]{Ulirsch_nonArchTeich}, one can extend the construction of a non-Archimedean skeleton from \cite{Thuillier_toroidal, ACP} to a possibly non-trivially valued base field $K$. We leave the details to the avid reader, since the statement of Theorem \ref{thm_skeletonvstropicalization} is not strictly used in the remainder of this article.

\section{Realizability of abelian harmonic covers}\label{section_realizability}







In this section, we return to the abelian case and fix a finite abelian group $\frakA$. We show that the $\frakA$-dilation datum of a harmonic $\frakA$-cover $\phi:\Ga'\to \Ga$ that is obtained by tropicalizing an algebraic $\frakA$-cover has a simple cohomological description. Conversely, we show that any harmonic $\frakA$-cover whose $\frakA$-dilation datum admits such a description comes from an algebraic $\frakA$-cover. This gives us an elementary necessary condition for realizability (see Corollary~\ref{cor:nobridges}), and other similar conditions can be readily found. 

We begin by giving the definition of realizability for weighted graphs and for tropical curves. 

\begin{definition} Let $k$ be an algebraically closed field. \label{def:realizability}

\begin{enumerate}
    \item An unramified harmonic $\frakA$-cover of weighted graphs $f:G'\to G$ is \emph{realizable} over $k$ if there exists an admissible $\frakA$-cover $X'_0\to X_0$ of stable nodal curves over $k$ whose dual harmonic $\frakA$-cover is $f$. 

    \item An unramified harmonic $\frakA$-cover of tropical curves $\phi:\Ga'\to \Ga$ is \emph{realizable} over $k$ if there exist a non-Archimedean field $K$ whose residue field is $k$ and a Galois $\frakA$-cover $F:X'\to X$ of smooth projective curves over $K$ such that $\phi$ is the tropicalization of $F$. 
    
\end{enumerate}

\end{definition}

\subsection{From Galois covers to extended homology} \label{sec:extended} Let $K$ be a non-Archimedean field with valuation ring $R$ and residue field $k$, whose characteristic $p$ is either zero or coprime to $\vert \frakA\vert$. Let $F\colon X'\to X$ be a finite $\frakA$-cover of smooth projective curves over $K$ (where $X'$ may be disconnected), which is ramified precisely at $n'$ marked ramification points $p'_1,\ldots, p'_{n'}\in X'$ over a collection of marked branch points $p_1,\ldots, p_n\in X$. Let $\calF\colon\calX'\to \calX$ be an extension of $X'\to X$ to a family of admissible $\frakA$-covers over $R$ (where we may have to replace $K$ by a finite extension, as above). Let $\phi\colon \Gamma_{X'}\rightarrow \Gamma_{X}$ be the induced tropical harmonic $\frakA$-cover with model $f:G_{\calX'}\to G_{\calX}$ (which depends on the choice of $\calF$ extending $F$).

Let $v\in V(G_{\calX})$ be a vertex, then the smooth locus $X^*_v$ of the irreducible component $X_v$ is a genus $g(v)$ curve over $k$ with $\val(v)$ punctures. The $\frakA$-cover $\calF^{-1}(X^*_v)\to X^*_v$ is determined by a monodromy representation $m_v:\pi_1^{\et}(X^*_v,x_0)\to \frakA$. Since $\frakA$ is abelian, the choice of base point is irrelevant, and the representation can be recorded by a tuple of elements of $\frakA$ in the following way. Let
\begin{equation*}
\Pi_{g(v),\val(v)}=\big\langle \alpha_1,\ldots,\alpha_{g(v)},\beta_1\ldots,\beta_{g(v)},\ga_1,\ldots,\ga_{\val(v)}\ \mid \ [\alpha_1,\beta_1]\cdots[\alpha_{g(v)},\beta_{g(v)}]\ga_1\cdots\ga_{\val(v)}=1\big\rangle 
\end{equation*}
be the fundamental group of a genus $g(v)$ Riemann surface with $\val(v)$ punctures, where the $\ga_j$ are small loops around the punctures. By a theorem of Grothendieck (see e.g.\ \cite[Theorem 4.9.1]{Szamuely}),  
the \'etale fundamental group $\pi_1^\et(X^*_v,x_0)$ is  the profinite completion of $\Pi_{g(v),\val(v)}$
when $p=0$ and the prime-to-$p$ profinite completion of $\Pi_{g(v),\val(v)}$ when $p>0$.
Since $\vert \frakA\vert$ is coprime to $p$, every continuous homomorphism $\pi_1^\et(X^*_v,x_0)\rightarrow \frakA$ (where $\frakA$ is equipped with the discrete topology) is uniquely determined by a homomorphism $\varphi:\Pi_{g(v),\val(v)}\rightarrow \frakA$ that factors as 
\begin{equation*}
\Pi_{g(v),\val(v)}\longrightarrow \pi_1^\et(X^*_v,x_0) \longrightarrow \frakA \ .
\end{equation*}
Hence the monodromy representation $m_v\colon \pi_1^{\et}(X^*_v,x_0)\to \frakA$ is uniquely determined by the images
\[
\xi(v)_i=\varphi(\alpha_i)\in \frakA\quad \textrm{and} \quad \xi(v)_{g(v)+i}=\varphi(\beta_i)\in\frakA\quad \textrm{for}\quad i=1,\ldots,g(v),
\]
of the $\alpha_i$ and $\beta_i$, which may be arbitrary, as well as the images
\[
\eta(h)=\varphi(\ga_j)\in \frakA\quad\textrm{for}\quad j=1,\ldots,\val (v),
\]
where $h\in H(G_{\calX})$ is the half-edge corresponding to the $j$-th puncture on $X_v$. We note that $\eta(h)$ acts by multiplication by a primitive $r$-th root of unity in an \'etale neighborhood of $p_{h'}$, where $r=d_{\phi}(h')$. The unique relation in the group $\Pi_{g(v),\val(v)}$ implies that the elements $\eta(h)$ satisfy
\[
\sum_{j=1}^{\val(v)}\varphi(\ga_j)=\sum_{h\in T_vG_{\calX}}\eta(h)=0.
\]
Furthermore, for each pair of nodes $e=\{h,h'\}$ we have $\eta(h)+\eta(h')=0$.


In other words, to an algebraic $\frakA$-cover $X'\to X$ we associate the following data on the graph $G_{\calX}$:

\begin{enumerate} 

\item An element $\eta(h)\in \frakA$ for each $h\in H(G_{\calX})$, so that $\eta(h)+\eta(h')=0$ for any edge $e=\{h,h'\}$ in $E(G_\calX)$ and $\sum_{h\in T_v G_\calX} \eta(h)=0$ for any vertex $v\in V(G_\calX)$.

\item An element $\xi(v)\in \frakA^{2g(v)}$ for every vertex $v\in V(G_\calX)$.

\end{enumerate}

The collection of all $\eta(h)$ is nothing but a class $\eta\in H_1(G_\calX,\frakA)$ in the simplicial homology of the graph $G_{\calX}$ with coefficients in $\frakA$. Similarly, each $\frakA^{2g(v)}$ can be thought of as the simplicial homology group (with coefficients in $\frakA$) of an infinitesimal genus $g(v)$ graph located at the vertex $v$. This motivates the following definition.

\begin{definition} Let $(G,g)$ be a weighted graph. The \emph{extended homology group} of $G$ with coefficients in $\frakA$ is the finite abelian group
\[
H^{\ext}_1(G,\frakA)=H_1(G,\frakA)\oplus \bigoplus_{v\in V(G)} \frakA^{2g(v)}.
\]
Given a harmonic $\frakA$-cover $f:G_{\calX'}\to G_{\calX}$ of weighted graphs that is the tropicalization of an algebraic $\frakA$-cover $F:X'\to X$, the datum $(\eta,\xi)\in H_1^{\ext}(G_{\calX},\frakA)$ defined above is called the \emph{$\frakA$-monodromy datum} associated to the cover. 

\end{definition}

We note that if $|\frakA|$ is coprime to $p$, then the group of continuous homomorphisms from $\pi_1^{\et}(X,x)$ to $\frakA$ is naturally identified with $H^1_{\et}(X,\frakA)$ (see Example 11.3 in \cite{Milne}), and hence the discussion above provides us with a natural homomorphism
 \begin{equation*}
H^1_\et(X,\frakA)\longrightarrow H^{\ext}_1(G_{\calX},\frakA).
\end{equation*}
Note that we pass from cohomology to homology, as $G_{\calX}$ is the dual graph of the special fiber $\calX_0$.

\subsection{Extended homology, dilation, and realizability}  We now observe that the $\frakA$-dilation datum of a harmonic $\frakA$-cover $f:G_{\calX'}\to G_{\calX}$ obtained by tropicalizing an algebraic $\frakA$-cover can be read off from the $\frakA$-monodromy datum $(\eta,\xi)$. Indeed, let $v'\in V(G_{\calX'})$ be a vertex mapping to $v=f(v')$. The restricted $\frakA$-cover $f^{-1}(X^*_v)\to X^*_v$ corresponds to the monodromy representation $m_v:\pi_1^{\et}(X_v^*,x_0)\to \frakA$, so an element of $\frakA$ preserves the irreducible component $X^*_{v'}\subseteq f^{-1}(X^*_v)$ (in other words, fixes $v'$) if and only if it is in the image of $m_v$, which is generated by the images of the $\alpha_i$, $\beta_i$, and $\gamma_j$. Similarly, the stabilizer of a node $p_{h'}\in \calX'_0$ mapping to $p_h\in \calX_0$ is generated by $\eta(h)\in \frakA$. Hence we give the following definition.


\begin{definition} Let $(G,g)$ be a weighted graph, and let $(\eta,\xi)\in H_1^{\ext}(G,\frakA)$ be an element of the extended homology group of $G$ with coefficients in $\frakA$. The \emph{associated $\frakA$-dilation datum} on $G$ is defined as follows:

\begin{enumerate} \item For a half-edge $h\in H(G)$, the dilation group $D(h)$ is the cyclic subgroup of $\frakA$ generated by the element $\eta(h)$.

\item For a vertex $v\in V(G)$, the dilation group $D(v)$ is the subgroup of $\frakA$ generated by $\eta(h)$ for all $h\in T_vG$ and by the entries of $\xi(v)$. 

\end{enumerate}

\end{definition}

We are now ready to state and prove our realizability criterion.

\begin{theorem} Let $k$ be an algebraically closed field whose characteristic is either zero or is relatively prime to $|\frakA|$.

\begin{enumerate}
    \item An unramified harmonic $\frakA$-cover $f:G'\to G$ of weighted graphs is realizable over $k$ if and only if the $\frakA$-dilation datum $D$ of $f$ is associated to some element $(\eta,\xi)\in H_1^{\ext}(G,\frakA)$ of the extended homology group of $G$ with coefficients in $\frakA$.

    \item An unramified harmonic $\frakA$-cover $\phi:\Ga'\to \Ga$ of tropical curves is realizable over $k$ if and only if it admits a realizable model $f:G'\to G$.
    
\end{enumerate}

\label{thm:main}
    
\end{theorem}

\begin{proof} Let $F_0:X'_0\to X_0$ be an admissible $\frakA$-cover of stable nodal curves over $k$, and let $f:G'\to G$ be the dual cover. The discussion in Section~\ref{sec:extended} produces an $\frakA$-monodromy datum $(\eta,\xi)\in H_1^{\ext}(G,\frakA)$ that defines the $\frakA$-dilation datum of $f$. Similarly, let $F:X'\to X$ be a $\frakG$-cover of smooth projective curves over a non-Archimedean field $K$ with residue field $k$, then the dual cover $\phi:\Ga_{X'}\to \Ga_X$ of tropical curves has a realizable model. 


Conversely, suppose that the $\frakA$-dilation datum $D$ of an unramified harmonic $\frakA$-cover $f:G'\to G$ is associated to an element $(\eta,\xi)\in H_1^{\ext}(G,\frakA)$. We reverse the procedure and construct an admissible $\frakA$-cover $F_0:X'_0\to X_0$ over $k$ tropicalizing to $f$, as follows. For each vertex $v\in V(G)$, choose a smooth $k$-curve $X^*_v$ of genus $g(v)$ with $|T_vG|$ punctures. The monodromy element $(\eta,\xi)$ induces a monodromy representation $m_v:\pi_1^{\et}(X^*_v,x_0)\to \frakA$ at each $v\in V(G)$ and hence an $\frakA$-cover of each $X^*_v$. We then glue these covers according to the incidence data of the graphs $G'$ and $G$ to obtain an admissible $\frakA$-cover $F_0:X'_0\to X_0$, where the $X_v$ are the irreducible components of $X_0$. Hence $f$ is realizable. 

If $f:G'\to G$ is the underlying model for a harmonic $\frakA$-cover $\phi:\Ga'\to \Ga$, then the admissible $\frakA$-cover $F_0:X'_0\to X_0$ can be smoothened to a family $\calF:\calX'\to \calX$ by the smoothness of the moduli space of admissible $\frakA$-covers over $\Spec \Z[\frac{1}{\vert \frakA\vert}]$ (see \cite[Theorem 3.0.2]{AbramovichCortiVistoli}). Alternatively, one may also use the smoothening result for harmonic covers of metrized curve complexes from \cite{ABBRI} and observe that their smoothing is compatible with group operations.
\end{proof}

This criterion can be used to establish elementary graph-theoretic restrictions on realizable harmonic $\frakA$-covers. First, we note that if the dilation group of any half-edge is not cyclic, then the cover is not realizable. Now let $G$ be a weighted graph with a bridge edge $e\in E(G)$. Any $\frakA$-monodromy datum $(\eta,\xi)\in H_1^{\ext}(G,\frakA)$ vanishes on $e$, so the dilation group of any realizable $\frakA$-cover is trivial along $e$. Hence we have established the following necessary condition for realizability.

\begin{corollary} \label{cor:nobridges}
Let $\phi\colon \Gamma'\rightarrow\Gamma$ be a harmonic $\frakA$-cover of metric graphs. If any bridge edge of $\Ga$ is dilated, then $f$ is not realizable.
\end{corollary}

In the next section, we show that, for most simple abelian groups, this condition is in fact sufficient. Similarly, if $G$ has a pair of parallel edges $e_1$ and $e_2$, then the dilation groups of any realizable $\frakA$-cover are equal on $e_1$ and $e_2$, since $\eta(e_2)=\pm \eta(e_1)$ for any $(\eta,\xi)\in H_1^{\ext}(G,\frakA)$.


\section{Cyclic covers and the nowhere-zero flow problem}\label{sec:nowherezeroflowproblem}

In this section, we discuss unramified harmonic $\frakA$-covers of a weighted graph $(G,g)$ without legs in the case where $\frakA=\ZZ/p\ZZ$ is a cyclic group of prime order. We show that the realizability problem for such covers is closely related to a classical problem in graph theory. 


Let $G$ be a graph and let $n\geq 2$. A \emph{nowhere-zero $n$-flow} on $G$ is an element $\eta\in H_1(G,\ZZ/n\ZZ)$ such that $\eta(e)\neq 0$ for all $e\in E(G)$.  The problem is to determine sufficient conditions for the existence of a nowhere-zero $n$-flow on $G$. It is clear that $G$ admits a nowhere-zero $n$-flow for any $n\geq 2$ only if $G$ has no bridges, and that it admits a nowhere-zero $2$-flow if and only if $\val(v)$ is even for all $v\in V(G)$. On the other hand, Seymour's $6$-flow Theorem \cite{Seymour_nowherezero6flows} states that any bridgeless graph admits a nowhere-zero $n$-flow for $n=6$, and it easily follows that this holds for any $n\geq 7$ as well. The intermediate cases $n=3$, $4$, and $5$, however, are not currently known. Tutte's conjecture states that every bridgeless graph has a nowhere-zero $5$-flow \cite[Conjecture II]{1954TutteChromatic}. 



Consider an unramified harmonic $\ZZ/p\ZZ$-cover $f:G'\to G$ of weighted graphs, where $p\geq 2$ is a prime number. Let $G_{\dil}\subseteq G$ be the dilation subgraph, consisting of those vertices and half-edges whose dilation group is $\ZZ/p\ZZ$. We now use the theory of nowhere-zero $n$-flows to show that the realizability of $f$ is determined by the structure of $G_{\dil}$.

\begin{theorem} Let $p$ be a prime number, and let $f:G'\to G$ be an unramified harmonic $\ZZ/p\ZZ$-cover of weighted graphs with no legs. If $p=2$, then $f$ is realizable. If $p\geq 7$, then $f$ is realizable if and only if the dilation subgraph $G_{\dil}\subseteq G$ has no bridges. If Tutte's conjecture holds, the same is true for $p=5$.
\end{theorem} 


\begin{proof} We recall the results of Example~\ref{ex:semistable}. For $v\in V(G_{\dil})$, denote its valency in $G_{\dil}$ by $d(v)$. We showed that the dilation subgraph $G_{\dil}\subseteq G$ is semistable: for each $v\in V(G_{\dil})$, either $d(v)\geq 2$ or $g(v)\geq 1$ (in addition, $d(v)$ is even if $p=2$). Unwrapping the definitions, Theorem~\ref{thm:main} implies that the cover $f:G'\to G$ is realizable if and only if there exists an element $(\eta,\xi)\in H_1^{\ext}(G,\ZZ/p\ZZ)$ satisfying the following conditions:

\begin{enumerate}
    \item $\eta(e)\neq 0$ if and only if $e\in E(G_{\dil})$.

    \item If $v\in V(G)$ is a vertex with no adjacent dilated edges, then $v\in V(G_{\dil})$ if and only if $\xi(v)\neq 0$.
    
\end{enumerate}

The first condition is always satisfied when $p=2$: setting $\eta(e)=1$ if and only if $e\in E(G_{\dil})$ gives a cycle $\eta\in H_1(G,\ZZ/2\ZZ)$, since $d(v)$ is even for all $v\in V(G_{\dil})$ (see \cite[Lemma 5.9]{2018JensenLen}). For $p\geq 7$, Seymour's theorem implies that we can find such an $\eta$ if and only if $G_{\dil}$ has no bridges, and Tutte's conjecture implies the same for $p=5$. The second condition, on the other hand, is trivially satisfied: if $v\in V(G_{\dil})$ is a dilated vertex with $d(v)=0$, then $g(v)\geq 1$ and hence we can pick $\xi(v)$ to be any nontrivial element of $(\ZZ/p\ZZ)^{2g(v)}$, and similarly we can set $\xi(v)=0$ for all $v\in V(G_{\dil})\backslash V(G)$. This completes the proof.
\end{proof}

We note that our results allow us to restate Tutte's conjecture in a purely algebraic form. Let $G$ be a bridgeless graph, and let $X_0$ be a nodal curve whose dual graph is $G$ (over any algebraically closed field of characteristic not equal to $5$). Suppose that $X_0$ admits an admissible $\ZZ/5\ZZ$-cover $X'_0\to X_0$ that is ramified at each node of $X_0$. The dual harmonic $\ZZ/5\ZZ$-cover $f:G'\to G$ has $G_{\dil}=G$, hence the $\ZZ/5\ZZ$-monodromy datum $(\eta,\xi)$ satisfies $\eta(e)\neq 0$ for all $e\in E(G)$.  Tutte's conjecture is  now equivalent to the following:

\begin{conjecture} Let $k$ be an algebraically closed field with $\characteristic k\neq 5$. Every nodal curve over $k$ with no separating nodes has an admissible $\ZZ/5\ZZ$-cover that is ramified at each node. 

\end{conjecture}

 \bibliographystyle{amsalpha}
\bibliography{biblio}{}

\providecommand{\bysame}{\leavevmode\hbox to3em{\hrulefill}\thinspace}
\providecommand{\MR}{\relax\ifhmode\unskip\space\fi MR }
\providecommand{\MRhref}[2]{%
  \href{http://www.ams.org/mathscinet-getitem?mr=#1}{#2}
}
\providecommand{\href}[2]{#2}
\begin{thebibliography}{KKMSD73}

\bibitem[AB15]{AminiBaker}
Omid Amini and Matthew Baker, \emph{Linear series on metrized complexes of
  algebraic curves}, Math. Ann. \textbf{362} (2015), no.~1-2, 55--106.

\bibitem[ABBR15a]{ABBRI}
Omid Amini, Matthew Baker, Erwan Brugall\'e, and Joseph Rabinoff, \emph{Lifting
  harmonic morphisms {I}: metrized complexes and {B}erkovich skeleta}, Res.
  Math. Sci. \textbf{2} (2015), Art. 7, 67.

\bibitem[ABBR15b]{ABBRII}
Omid Amini, Matthew Baker, Erwan Brugall\'{e}, and Joseph Rabinoff,
  \emph{Lifting harmonic morphisms {II}: {T}ropical curves and metrized
  complexes}, Algebra Number Theory \textbf{9} (2015), no.~2, 267--315.

\bibitem[ACP15]{ACP}
Dan Abramovich, Lucia Caporaso, and Sam Payne, \emph{The tropicalization of the
  moduli space of curves}, Ann. Sci. \'{E}c. Norm. Sup\'{e}r. (4) \textbf{48}
  (2015), no.~4, 765--809.

\bibitem[ACV03]{AbramovichCortiVistoli}
Dan Abramovich, Alessio Corti, and Angelo Vistoli, \emph{Twisted bundles and
  admissible covers}, Comm. Algebra \textbf{31} (2003), no.~8, 3547--3618,
  Special issue in honor of Steven L. Kleiman.

\bibitem[AV02]{AbramovichVistoli}
Dan Abramovich and Angelo Vistoli, \emph{Compactifying the space of stable
  maps}, J. Amer. Math. Soc. \textbf{15} (2002), no.~1, 27--75.

\bibitem[Bas93]{Bass}
Hyman Bass, \emph{Covering theory for graphs of groups}, J. Pure Appl. Algebra
  \textbf{89} (1993), no.~1-2, 3--47.

\bibitem[BBC17]{2017BologneseBrandtChua}
Barbara Bolognese, Madeline Brandt, and Lynn Chua, \emph{From curves to
  tropical {J}acobians and back}, Combinatorial algebraic geometry, Fields
  Inst. Commun., vol.~80, Fields Inst. Res. Math. Sci., Toronto, ON, 2017,
  pp.~21--45.

\bibitem[BH20]{2017BrandtHelminck}
Madeline Brandt and Paul~Alexander Helminck, \emph{Tropical superelliptic
  curves}, Adv. Geom. \textbf{20} (2020), no.~4, 527--551.

\bibitem[BN09]{2009BakerNorine}
Matthew Baker and Serguei Norine, \emph{Harmonic morphisms and hyperelliptic
  graphs}, Int. Math. Res. Not. IMRN (2009), no.~15, 2914--2955.

\bibitem[BR11]{BertinRomagny}
Jos\'{e} Bertin and Matthieu Romagny, \emph{Champs de {H}urwitz}, M\'{e}m. Soc.
  Math. Fr. (N.S.) (2011), no.~125-126, 219.

\bibitem[Cap14]{Caporaso_gonality}
Lucia Caporaso, \emph{Gonality of algebraic curves and graphs}, Algebraic and
  complex geometry, Springer Proc. Math. Stat., vol.~71, Springer, Cham, 2014,
  pp.~77--108.

\bibitem[CF17]{ChiodoFarkas}
Alessandro Chiodo and Gavril Farkas, \emph{Singularities of the moduli space of
  level curves}, J. Eur. Math. Soc. (JEMS) \textbf{19} (2017), no.~3, 603--658.

\bibitem[Cha13]{2013Chan}
Melody Chan, \emph{Tropical hyperelliptic curves}, J. Algebraic Combin.
  \textbf{37} (2013), no.~2, 331--359.

\bibitem[CLRW22]{creech2022prym}
Steven Creech, Yoav Len, Caelan Ritter, and Derek Wu,
  \emph{Prym--{B}rill--{N}oether loci of special curves}, International
  Mathematics Research Notices \textbf{2022} (2022), no.~4, 2688--2728.

\bibitem[CM16]{CavalieriMiles}
Renzo Cavalieri and Eric Miles, \emph{Riemann surfaces and algebraic curves},
  London Mathematical Society Student Texts, vol.~87, Cambridge University
  Press, Cambridge, 2016, A first course in Hurwitz theory.

\bibitem[CMP20]{CaporasoMeloPacini}
Lucia Caporaso, Margarida Melo, and Marco Pacini, \emph{Tropicalizing the
  moduli space of spin curves}, Selecta Math. (N.S.) \textbf{26} (2020), no.~1,
  Paper No. 16, 44.

\bibitem[CMR16]{CavalieriMarkwigRanganathan_tropadmissiblecovers}
Renzo Cavalieri, Hannah Markwig, and Dhruv Ranganathan, \emph{Tropicalizing the
  space of admissible covers}, Math. Ann. \textbf{364} (2016), no.~3-4,
  1275--1313.

\bibitem[Cor11]{2011Corry}
Scott Corry, \emph{Genus bounds for harmonic group actions on finite graphs},
  Int. Math. Res. Not. IMRN (2011), no.~19, 4515--4533.

\bibitem[Cor12]{2012Corry}
\bysame, \emph{Harmonic {G}alois theory for finite graphs},
  Galois-{T}eichm\"{u}ller theory and arithmetic geometry, Adv. Stud. Pure
  Math., vol.~63, Math. Soc. Japan, Tokyo, 2012, pp.~121--140.

\bibitem[Cor15]{2015Corry}
\bysame, \emph{Maximal harmonic group actions on finite graphs}, Discrete Math.
  \textbf{338} (2015), no.~5, 784--792.

\bibitem[Eke95]{Ekedahl_Hurwitzboundary}
Torsten Ekedahl, \emph{Boundary behaviour of {H}urwitz schemes}, The moduli
  space of curves ({T}exel {I}sland, 1994), Progr. Math., vol. 129,
  Birkh\"{a}user Boston, Boston, MA, 1995, pp.~173--198.

\bibitem[Gal19a]{2019GaleottiB}
Mattia Galeotti, \emph{Birational geometry of moduli of curves with an
  {${S}_3$}-cover}, arXiv:1905.03487 [math] (2019).

\bibitem[Gal19b]{2019GaleottiA}
\bysame, \emph{Moduli of {$G$}-covers of curves: geometry and singularities},
  arXiv:1905.02889 [math] (2019).

\bibitem[GZ23]{2023GhoshZakharov}
Arkabrata Ghosh and Dmitry Zakharov, \emph{The {P}rym variety of a dilated
  double cover of metric graphs}, arXiv preprint arXiv:2303.03904 (2023).

\bibitem[Hel17]{2017Helminck}
Paul~Alexander Helminck, \emph{Tropicalizing abelian covers of algebraic
  curves}, arXiv:1703.03067 [math] (2017).

\bibitem[Hel21]{Helminck_skeletalfiltrations}
\bysame, \emph{Skeletal filtrations of the fundamental group of a
  non-archimedean curve}, February 2021, arXiv:1808.03541 [math].

\bibitem[JKK05]{JarvisKaufmannKimura}
Tyler~J. Jarvis, Ralph Kaufmann, and Takashi Kimura, \emph{Pointed admissible
  {$G$}-covers and {$G$}-equivariant cohomological field theories}, Compos.
  Math. \textbf{141} (2005), no.~4, 926--978.

\bibitem[JL18]{2018JensenLen}
David Jensen and Yoav Len, \emph{Tropicalization of theta characteristics,
  double covers, and {P}rym varieties}, Sel. Math. New Ser. \textbf{24} (2018),
  1391--1410.

\bibitem[KKMSD73]{KKMSD}
George Kempf, Finn~Faye Knudsen, David Mumford, and Bernard Saint-Donat,
  \emph{Toroidal embeddings. {I}}, Lecture Notes in Mathematics, Vol. 339,
  Springer-Verlag, Berlin-New York, 1973.

\bibitem[Len17]{2017Len}
Yoav Len, \emph{Hyperelliptic graphs and metrized complexes}, Forum Math. Sigma
  \textbf{5} (2017), e20, 15.

\bibitem[Len22]{2022Len}
\bysame, \emph{Chip-firing games, {J}acobians, and {P}rym varieties}, arXiv
  preprint arXiv:2210.14060 (2022).

\bibitem[LU21]{LenUlirsch}
Yoav Len and Martin Ulirsch, \emph{Skeletons of {P}rym varieties and
  {B}rill-{N}oether theory}, Algebra Number Theory \textbf{15} (2021), no.~3,
  785--820.

\bibitem[LZ22]{2022LenZakharov}
Yoav Len and Dmitry Zakharov, \emph{Kirchhoff’s theorem for {P}rym
  varieties}, Forum of Mathematics, Sigma \textbf{10} (2022), e11.

\bibitem[Mil13]{Milne}
James~S. Milne, \emph{Lectures on etale cohomology (v2.21)}, 2013, Available at
  www.jmilne.org/math/, p.~202.

\bibitem[Mir95]{Miranda}
Rick Miranda, \emph{Algebraic curves and {R}iemann surfaces}, Graduate Studies
  in Mathematics, vol.~5, American Mathematical Society, Providence, RI, 1995.

\bibitem[Pan16]{2016Panizzut}
Marta Panizzut, \emph{Theta characteristics of hyperelliptic graphs}, Arch.
  Math. (Basel) \textbf{106} (2016), no.~5, 445--455.

\bibitem[PP06]{PervovaPetronio_HurwitzexistenceI}
Ekaterina Pervova and Carlo Petronio, \emph{On the existence of branched
  coverings between surfaces with prescribed branch data. {I}}, Algebr. Geom.
  Topol. \textbf{6} (2006), 1957--1985.

\bibitem[RW06]{RomagnyWewers}
Matthieu Romagny and Stefan Wewers, \emph{Hurwitz spaces}, Groupes de {G}alois
  arithm\'{e}tiques et diff\'{e}rentiels, S\'{e}min. Congr., vol.~13, Soc.
  Math. France, Paris, 2006, pp.~313--341.

\bibitem[RZ22]{2022RoehrleZakharov}
Felix R{\"o}hrle and Dmitry Zakharov, \emph{The tropical $n$-gonal
  construction}, arXiv preprint arXiv:2210.02267 (2022).

\bibitem[Sa{\"i}97]{Saidi_graphofgroups}
Mohamed Sa{\"i}di, \emph{Rev\^{e}tements mod\'{e}r\'{e}s et groupe fondamental
  de graphe de groupes}, Compositio Math. \textbf{107} (1997), no.~3, 319--338.

\bibitem[Ser80]{Serre_trees}
Jean-Pierre Serre, \emph{Trees}, Springer-Verlag, Berlin-New York, 1980,
  Translated from the French by John Stillwell.

\bibitem[Sey81]{Seymour_nowherezero6flows}
P.~D. Seymour, \emph{Nowhere-zero {$6$}-flows}, J. Combin. Theory Ser. B
  \textbf{30} (1981), no.~2, 130--135.

\bibitem[Son19]{Song_Ginvariantlinearsystems}
JuAe Song, \emph{Galois quotients of metric graphs and invariant linear
  systems}, arXiv:1901.09172 [math] (2019).

\bibitem[SvZ20]{SchmittvanZelm}
Johannes Schmitt and Jason van Zelm, \emph{Intersections of loci of admissible
  covers with tautological classes}, Selecta Math. (N.S.) \textbf{26} (2020),
  no.~5, Paper No. 79, 69.

\bibitem[Sza09]{Szamuely}
Tam\'{a}s Szamuely, \emph{Galois groups and fundamental groups}, Cambridge
  Studies in Advanced Mathematics, vol. 117, Cambridge University Press,
  Cambridge, 2009.

\bibitem[Thu07]{Thuillier_toroidal}
Amaury Thuillier, \emph{G\'{e}om\'{e}trie toro\"{i}dale et g\'{e}om\'{e}trie
  analytique non archim\'{e}dienne. {A}pplication au type d'homotopie de
  certains sch\'{e}mas formels}, Manuscripta Math. \textbf{123} (2007), no.~4,
  381--451.

\bibitem[Tut54]{1954TutteChromatic}
William Tutte, \emph{A contribution to the theory of chromatic polynomials},
  Canad. J. Math. (1954), no.~6, 80--91.

\bibitem[Uli17]{Ulirsch_tropisquot}
Martin Ulirsch, \emph{Tropicalization is a non-{A}rchimedean analytic stack
  quotient}, Math. Res. Lett. \textbf{24} (2017), no.~4, 1205--1237.

\bibitem[Uli21]{Ulirsch_nonArchTeich}
\bysame, \emph{A non-{A}rchimedean analogue of {T}eichm\"{u}ller space and its
  tropicalization}, Selecta Math. (N.S.) \textbf{27} (2021), no.~3, Paper No.
  39, 34.

\bibitem[UZ19]{ulirsch2019tropical}
Martin Ulirsch and Dmitry Zakharov, \emph{Tropical double ramification loci},
  arXiv:1910.01499 (2019).

\end{thebibliography}

\end{document}